\documentclass[a4paper,11pt,oneside]{amsart}
\pdfoutput=1

\usepackage[utf8]{inputenc}
\usepackage{bm}
\usepackage{mathtools,amssymb}
\usepackage{hyperref}
\usepackage{xcolor,graphicx}
\usepackage{mathrsfs}
\usepackage[shortlabels]{enumitem}
\usepackage{lineno}
\usepackage{amsmath,eufrak}
\usepackage{enumitem}
\usepackage{amsthm} 
\usepackage{fullpage} 

\allowdisplaybreaks

\mathtoolsset{showonlyrefs}

\graphicspath{{images/}}

\newtheorem{theorem}{Theorem}[section]
\newtheorem{lemma}[theorem]{Lemma}
\newtheorem{definition}[theorem]{Definition}

\newtheorem{remark}[theorem]{Remark}
\newtheorem{proposition}[theorem]{Proposition}

\title[The higher order fractional Calder\'on problem: uniqueness]{The higher order fractional Calder\'on problem for linear local operators: uniqueness}

\author{Giovanni Covi}
\address{Department of Mathematics and Statistics, University of Jyv\"askyl\"a, Jyv\"askyl\"a, Finland}
\curraddr{Institut fur Angewandte Mathematik, Ruprecht-Karls-Universit\"at Heidelberg, Im Neuenheimer Feld 205, 69120 Heidelberg, Germany}
\email{giovanni.covi@uni-heidelberg.de}

\author{Keijo M\"onkk\"onen}
\address{Department of Mathematics and Statistics, University of Jyv\"askyl\"a, Jyv\"askyl\"a, Finland}
\email{kematamo@student.jyu.fi}

\author{Jesse Railo}
\address{Seminar for Applied Mathematics, Department of Mathematics, ETH Zurich, Z\"urich, Switzerland}
\curraddr{Department of Pure Mathematics and Mathematical Statistics, University of
Cambridge, Cambridge CB3 0WB, UK}
\email{jr891@cam.ac.uk}
\author{Gunther Uhlmann}
\address{Department  of  Mathematics,  University  of  Washington, Seattle, USA  /  Jockey  Club  Institute  for  Advanced Study, HKUST, Hong Kong}
\email{gunther@math.washington.edu}
\date{\today}

\newcommand{\R}{{\mathbb R}}
\newcommand{\Z}{{\mathbb Z}}
\newcommand{\N}{{\mathbb N}}

\newcommand{\der}{{\mathrm d}}
\newcommand{\id}{\mathrm{Id}}

\newcommand{\schwartz}{\mathscr{S}}

\newcommand{\tempered}{\mathscr{S}^{\prime}}

\newcommand{\fraclaplace}{(-\Delta)^s}
\newcommand{\fourier}{\mathcal{F}}
\newcommand{\ifourier}{\mathcal{F}^{-1}}

\newcommand{\dimens}{n}
\newcommand{\norm}[1]{\left\lVert #1 \right\rVert}
\newcommand{\abs}[1]{\left\lvert #1 \right\rvert}
\newcommand{\aabs}[1]{\left\lVert #1 \right\rVert}
\newcommand{\ip}[2]{\left\langle #1,#2 \right\rangle}
\DeclareMathOperator{\spt}{spt}

\begin{document}

\maketitle
\begin{abstract}
We study an inverse problem for the fractional Schr\"odinger equation (FSE) with a local perturbation by a linear partial differential operator (PDO) of order smaller than the one of the fractional Laplacian. We show that one can uniquely recover the coefficients of the PDO from the exterior Dirichlet-to-Neumann (DN) map associated to the perturbed FSE. This is proved for two classes of coefficients: coefficients which belong to certain spaces of Sobolev multipliers and coefficients which belong to fractional Sobolev spaces with bounded derivatives. Our study generalizes recent results for the zeroth and first order perturbations to higher order perturbations. 
\end{abstract}

\section{Introduction}\label{intro}
Let $s\in\R^+\setminus\Z$, $\Omega\subset\R^\dimens$ a bounded open set where $n\geq 1$, $\Omega_e=\R^\dimens\setminus\overline{\Omega}$ its exterior and $P(x, D)$ a linear partial differential operator (PDO) of order $m\in\N$
\begin{equation}
P(x, D)=\sum_{\abs{\alpha}\leq m}a_{\alpha}(x)D^{\alpha}
\end{equation}
where the coefficients $a_\alpha=a_\alpha(x)$ are functions defined in $\Omega$. We study a nonlocal inverse problem for the perturbed fractional Schr\"odinger equation
\begin{equation}
\label{eq:fractionalpdo}
\begin{cases}
\fraclaplace u+P(x, D)u=0 \ \text{in} \ \Omega \\
u=f \ \text{in} \ \Omega_e
\end{cases}
\end{equation}
where $\fraclaplace$ is a nonlocal pseudo-differential operator $\fraclaplace u=\ifourier(\abs{\cdot}^{2s}\hat{u})$ in contrast to the local operator $P(x, D)$. In the inverse problem, one aims to recover the local operator $P$ from the associated Dirichlet-to-Neumann map.

We always assume that $0$ is not a Dirichlet eigenvalue of the operator $(\fraclaplace+P(x, D))$, i.e.
\begin{equation}
\text{If} \ u\in H^s(\R^\dimens) \ \text{solves} \ (\fraclaplace+P(x, D))u=0 \ \text{in} \ \Omega \ \text{and} \ u|_{\Omega_e}=0, \ \text{then} \ u=0.
\end{equation}
Our data for the inverse problem is the exterior Dirichlet-to-Neumann (DN) map  $\Lambda_P\colon H^s(\Omega_e)\rightarrow (H^s(\Omega_e))^*$ which maps Dirichlet exterior values to a nonlocal version of the Neumann boundary value (see section \ref{sec:preliminaries} and \ref{subsec-WP-singular}). The main question that we study in this article is whether the exterior DN map $\Lambda_P$ determines uniquely the coefficients $a_\alpha$ in $\Omega$. In other words, does $\Lambda_{P_1}=\Lambda_{P_2}$ imply that $a_{1, \alpha}=a_{2, \alpha}$ in $\Omega$ for all $\abs{\alpha}\leq m$? We prove that the answer is positive under certain restrictions on the coefficients $a_\alpha$ and the order of the PDOs.

This gives a positive answer to the uniqueness problem \cite[Question 7.5]{CMR20} posed by the first three authors in a previous work. The precise statement in \cite{CMR20} asks to prove uniqueness for the higher order fractional Calderón problem in the case of a bounded domain with smooth boundary and PDOs with smooth coefficients (up to the boundary). The positive answer to this question follows from theorem \ref{main-theorem-bounded}. The study of the fractional Calderón problem was initiated by Ghosh, Salo and Uhlmann in the work \cite{GSU20} where the uniqueness for the associated inverse problem is proved when $m=0$, $s \in (0,1)$ and $a_0 \in L^\infty(\Omega)$. 

We briefly note that by Peetre's theorem any linear operator $L\colon C_c^{\infty}(\Omega)\rightarrow C_c^\infty(\Omega)$ which does not increase supports, i.e. $\spt(Lf)\subset\spt(f)$ for all $f\in C_c^\infty(\Omega)$, is in fact a differential operator~\cite{MI-refinement-of-peetres-theorem} (see also the original work \cite{PEE-peetres-theorem}). Therefore our results apply to any local operator satisfying such properties and it is enough to study PDOs only. For a more general formulation of Peetre's theorem on the level of vector bundles, see
~\cite{NS-peetre-revisited}.

\subsection{Main results}
We denote by $M(H^{s-|\alpha|}\rightarrow H^{-s})$ the space of all bounded Sobolev multipliers between the Sobolev spaces $H^{s-|\alpha|}(\R^\dimens)$ and $ H^{-s}(\R^\dimens)$. We denote by $M_0(H^{s-|\alpha|}\rightarrow H^{-s}) \subset M(H^{s-|\alpha|}\rightarrow H^{-s})$ the space of bounded Sobolev multipliers that can be approximated by smooth compactly supported functions in the multiplier norm of $M(H^{s-|\alpha|}\rightarrow H^{-s})$. We also write $H^{r, \infty}(\Omega)$ for the local Bessel potential space with bounded derivatives. See section \ref{sec:preliminaries} for more detailed definitions. 

Our first theorem is a generalization of \cite[Theorem 1.1]{RS-fractional-calderon-low-regularity-stability} which considered the case $m=0$ with $s \in (0,1)$. It also generalizes \cite[Theorem 1.5]{CMR20} which considered the higher order case $s \in \R^+ \setminus \Z$ when $m=0$.

\begin{theorem}
\label{main-theorem-singular}
Let $\Omega \subset \mathbb R^n$ be a bounded open set where $n\geq 1$. Let $s \in \mathbb R^+ \setminus \mathbb Z$ and $m\in \mathbb N$ be such that $2s > m$. Let $$P_j = \sum_{|\alpha|\leq m} a_{j,\alpha} D^\alpha,\quad j = 1, 2,$$ be linear PDOs of order $m$ with coefficients $a_{j,\alpha} \in M_0(H^{s-|\alpha|}\rightarrow H^{-s})$. Given any two open sets $W_1, W_2 \subset \Omega_e$, suppose that the exterior DN maps $\Lambda_{P_i}$ for the equations
$((-\Delta)^s + P_j)u = 0$ in $\Omega$ satisfy 
$$ \Lambda_{P_1}f|_{W_2} = \Lambda_{P_2}f|_{W_2}$$
for all $f \in C^\infty_c(W_1)$. Then $P_1|_{\Omega} = P_2|_{\Omega}$.
\end{theorem}

In theorem \ref{main-theorem-singular} one can pick the lower order coefficients ($\abs{\alpha}< s$) from $L^p(\Omega)$ for high enough~$p$ (especially from $L^\infty(\Omega)$) and higher order coefficients ($s<\abs{\alpha}<2s$) from the closure of $C^\infty_c(\Omega)$ in $H^{r, \infty}(\Omega)$ for certain values of $r\in\R$. In the following propositions, which are proved in Section 2, we give more examples of Sobolev spaces which belong to the space of multipliers $M_0(H
^{s-\abs{\alpha}}\rightarrow H^{-s})$:
\begin{proposition}
\label{lemma:generalcoefficients}
Let $\Omega\subset\R^\dimens$ be an open set and let $t\in\R$ and $r\in\R$ be such that $t>\max\{0, r\}$. The following inclusions hold:
\begin{enumerate}[(i)]
    \item\label{item1} $\widetilde{H}^{r^\prime, \infty}(\Omega)\subset M_0(H^{-r}\rightarrow H^{-t})$ whenever $r^\prime\geq\max\{0, r\}$.
    \medskip
    \item\label{item2} $H^{r', \infty}_0(\Omega)\subset M_0(H^{-r}\rightarrow H^{-t})$ whenever $r'\geq\max\{0, r\}$ with $r'\notin\{\frac{1}{2}, \frac{3}{2}, \frac{5}{2}, \dotso\}$ and $\Omega$ is a Lipschitz domain.
    \medskip
    \item\label{item3} $\widetilde{H}^{r'}(\Omega)\subset M_0(H^{-r}\rightarrow H^{-t})$ whenever $r'\geq t$ and $r'>n/2$. The same holds true for $H^{r'}_{\overline{\Omega}}(\R^\dimens)$ if $\Omega$ is a Lipschitz domain, and for $H^{r'}_0(\Omega)$ when $\Omega$ is a Lipschitz domain and $r'\notin\{\frac{1}{2}, \frac{3}{2}, \frac{5}{2}, \dotso\}$.
    \medskip
\end{enumerate}
\end{proposition}

Note that the assumptions in theorem \ref{main-theorem-singular} satisfy the conditions of proposition~\ref{lemma:generalcoefficients} since then $r=\abs{\alpha}-s$ and $t=s$. The next proposition gives examples of spaces of lower order coefficients ($\abs{\alpha}\leq s$):

\begin{proposition}
\label{lemma:lowerordercoefficients}
Let $\Omega\subset\R^\dimens$ be an open set and $t>0$. The following inclusions hold:
\begin{enumerate}[(i)]
 \item\label{item4} $L^p(\Omega)\subset M_0(H^0\rightarrow H^{-t})$ whenever $2\leq p<\infty$ and $p>n/t$. Especially, if $\Omega$ is bounded, then $L^\infty (\Omega)\subset M_0(H^0\rightarrow H^{-t})$.
    \medskip
    \item\label{item5} $\widetilde{H}^r(\Omega)\subset M_0(H^0\rightarrow H^{-t})$ whenever $r\geq 0$ and $r>n/2-t$. The same holds true for $H^r_{\overline{\Omega}}(\R^\dimens)$ if $\Omega$ is a Lipschitz domain, and for $H^r_0(\Omega)$ when $\Omega$ is Lipschitz domain and $r\notin\{\frac{1}{2}, \frac{3}{2}, \frac{5}{2}, \dotso\}$.
    \end{enumerate}
\end{proposition}

As mentioned above, we put $t=s>0$ in theorem \ref{main-theorem-singular} and the condition in proposition~\ref{lemma:lowerordercoefficients} is satisfied. Note that under the assumption $\abs{\alpha}\leq s$ we have $M_0(H^0\rightarrow H^{-s})\subset M_0(H^{s-\abs{\alpha}}\rightarrow H^{-s})$. Hence we can choose the lower order coefficients from a less regular space in theorem~\ref{main-theorem-singular} (compare to proposition~\ref{lemma:generalcoefficients}).
We also note that when $\abs{\alpha}=0$ the space of multipliers $M_0(H
^s\rightarrow H^{-s})$ coincides with the one studied in~\cite{RS-fractional-calderon-low-regularity-stability}.

It follows from lemma \ref{lemma-properties-of-multipliers} that the space of multipliers is trivial for higher order operators, i.e. $M(H^{s-|\alpha|}\rightarrow H^{-s}) = \{0\}$ when $s-\abs{\alpha} < -s$. It would be possible to state theorem \ref{main-theorem-singular} for higher order PDOs, but that forces $a_\alpha = 0$ for all $\abs{\alpha} > 2s$. For this reason we only consider PDOs whose order is $m <2s$. See lemma \ref{lemma-properties-of-multipliers} and the related remarks for more details.

Our second theorem generalizes \cite[Theorem 1.1]{CLR18} and \cite[Theorem 1.1]{GSU20}, where similar results are proved when  $m=0,1$ and $s \in (0,1)$. It also generalizes \cite[Theorem 1.5]{CMR20} where the case $m=0$ and $s\in\R^+\setminus\Z$ was studied.

\begin{theorem}
\label{main-theorem-bounded} Let $\Omega \subset \mathbb R^n$ be a bounded Lipschitz domain where $n\geq 1$. Let $s \in \mathbb R^+ \setminus \mathbb Z$ and $m\in \mathbb N$ be such that $2s > m$. Let $$P_j (x, D) = \sum_{|\alpha|\leq m} a_{j,\alpha}(x) D^\alpha,\quad j = 1, 2,$$ be a linear PDOs of order $m$ with coefficients $a_{j,\alpha} \in H^{r_\alpha,\infty}(\Omega)$ where
\begin{align}\label{r-alpha-conditions}
    r_\alpha := \Bigg\{ \begin{matrix*} 
    0 & \mbox{if} & |\alpha|-s<0, \\ 
    |\alpha|-s+\delta & \mbox{if} & |\alpha|-s \in \{ 1/2, 3/2, ... \}, \\
    |\alpha|-s & \mbox{if} & \mbox{otherwise} \\
    \end{matrix*} \Bigg. 
\end{align}

\noindent for any fixed $\delta >0$. Given any two open sets $W_1, W_2 \subset \Omega_e$, suppose that the exterior DN maps $\Lambda_{P_i}$ for the equations
$((-\Delta)^s + P_j (x, D))u = 0$ in $\Omega$ satisfy 
$$ \Lambda_{P_1}f|_{W_2} = \Lambda_{P_2}f|_{W_2}$$
for all $f \in C^\infty_c(W_1)$. Then $P_1(x, D) = P_2(x, D)$.
\end{theorem}

Our first theorem is formulated for general bounded open sets and the second theorem for Lipschitz domains. The difference arises in the proof of the well-posedness of the forward problem. We note that theorem \ref{main-theorem-bounded} holds for coefficients $a_\alpha$ which are smooth up to the boundary ($a_\alpha=g|_\Omega$ where $g\in C^\infty(\R^\dimens)$). The conditions \eqref{r-alpha-conditions} imply that one can choose $a_\alpha \in L^\infty(\Omega)$ for every $\alpha$ such that $|\alpha|<s$. The case $|\alpha| = s$ never happens, as $s$ is assumed not to be an integer. If $|\alpha| > s$, we have $a_\alpha \in H^{|\alpha|-s,\infty}(\Omega)$ when $|\alpha|-s \not\in \{ 1/2, 3/2, ... \}$. Thus the conditions \eqref{r-alpha-conditions} coincide with \cite{CLR18,GSU20} when $m=0,1$ and $s \in (0,1)$.

Our article is roughly divided into two parts. The first part of the article (theorem \ref{main-theorem-singular} and section \ref{section-main-sigular}) generalizes the study of the uniqueness problem for singular potentials in \cite{RS-fractional-calderon-low-regularity-stability} and the second part (theorem \ref{main-theorem-bounded} and section \ref{section-main-bounded}) generalizes the uniqueness problem for bounded first order perturbations in \cite{CLR18}.

The approach to prove theorems~\ref{main-theorem-singular} and \ref{main-theorem-bounded} is the following. First one shows that the forward problem is well-posed and the corresponding bilinear forms are bounded. This leads to the boundedness of the exterior DN maps and an Alessandrini identity. By a unique continuation property of the higher order fractional Laplacian one obtains a Runge approximation property for equation~\eqref{eq:fractionalpdo}. Using the Runge approximation and the Alessandrini identity for suitable test functions one proves the uniqueness of the inverse problem.

\subsection{On the earlier literature} Equation~\eqref{eq:fractionalpdo} and theorems~\ref{main-theorem-singular} and \ref{main-theorem-bounded} are related to the Calder\'on problem for the fractional Schr\"odinger equation first introduced in~\cite{GSU20}. There one tries to uniquely recover the potential $q$ in $\Omega$ by doing measurements in the exterior $\Omega_e$. This is a nonlocal (fractional) counterpart of the classical Calder\'on problem arising in electrical impedance tomography, where one obtains information about the electrical properties of some bounded domain by doing voltage and current measurements on the boundary~\cite{UHL-electrical-impedance-tomography, UH-inverse-problems-seeing-the-unseen}.
In~\cite{RS-fractional-calderon-low-regularity-stability} the study of the fractional Calder\'on problem is extended for ``rough" potentials $q$, i.e. potentials which are in general bounded Sobolev multipliers. First order perturbations were studied in~\cite{CLR18} assuming that the fractional part dominates the equation, i.e. $s\in (1/2, 1)$, and that the perturbations have bounded fractional derivatives. A higher order version ($s\in\R^+\setminus\Z$) of the fractional Calderón problem was introduced and studied in~\cite{CMR20}. These three articles \cite{CLR18, CMR20, RS-fractional-calderon-low-regularity-stability} motivate the study of higher order (rough) perturbations to the fractional Laplacian $\fraclaplace$ in equation \eqref{eq:fractionalpdo}. The natural restriction for the order of
~$P(x, D)$ in theorems
~\ref{main-theorem-singular} and \ref{main-theorem-bounded} is then $2s>m$, so that the fractional part governs the equation~\eqref{eq:fractionalpdo}. 

The fractional Calder\'on problem for $s\in (0, 1)$ has been studied in many settings. We refer to the survey
~\cite{Sal17} for a more detailed treatment. 
In the work~\cite{RS-fractional-calderon-low-regularity-stability} stability was proved for singular potentials, and in \cite{RS18} the related exponential instability was shown. The fractional Calder\'on problem has also been solved under single measurement~\cite{GRSU-fractional-calderon-single-measurement}.
The perturbed equation is related to the fractional magnetic Schr\"odinger equation which is studied in~\cite{CO-magnetic-fractional-schrodinger, LILI-semilinear-magnetic, LI-fractional-magnetic, LILI-fractional-magnetic-calderon}. See also \cite{BGU-lower-order-nonlocal-perturbations} for a fractional Schr\"odinger equation with a lower order nonlocal perturbation. Other variants of the fractional Calder\'on problem include semilinear fractional (magnetic) Schr\"odinger equation
~\cite{LL19, LL-fractional-semilinear-problems, LILI-semilinear-magnetic, LILI-fractional-magnetic-calderon}, fractional heat equation~\cite{LLR19,ruland2019quantitative} and fractional conductivity equation~\cite{Co18} (see also \cite{CLL19, GLX-calderon-nonlocal-elliptic-operators} for equations arising from a nonlocal Schr\"odinger-type elliptic operator). In the recent work~\cite{CMR20}, the first three authors of this article studied higher order versions ($s\in\R^+\setminus\Z$) of the fractional Calder\'on problem and proved uniqueness for the Calderón problem for the fractional magnetic Schr\"odinger equation (up to a gauge). This article continues these studies by showing uniqueness for the fractional Schr\"odinger equation with higher order perturbations and gives positive answer to a question posed in
~\cite[Question 7.5]{CMR20}.

\subsection{Examples of fractional models in the sciences}
Equations involving fractional Laplacians like~\eqref{eq:fractionalpdo} have applications in mathematics and natural sciences. Fractional Laplacians appear in the study of anomalous and nonlocal diffusion, and these diffusion phenomena can be used in many areas such as continuum mechanics, graph theory and ecology just to mention a few~\cite{VMRTM-nonlocal-diffusion-problems, BV-nonlocal-diffusion-applications, DGLZ2012, MV-nonlocal-diffusion-population-competition, RO-nonlocal-elliptic-equations-bounded-domains}. Another place where the fractional counterpart of the classical Laplacian naturally shows up is the formulation of fractional quantum mechanics~\cite{La00, LAS-fractional-schrodinger-equation, LA-fractional-quantum-mechanics}. See \cite{RS15-higher-order-frac} and references therein for possible applications of higher order fractional Laplacians. For more applications of fractional mathematical models, see
~\cite{BV-nonlocal-diffusion-applications, KM-random-walks-quide-anomalous-diffusion} and the references therein.

\subsection{The organization of the article}
In section \ref{sec:preliminaries} we introduce the notation and give preliminaries on Sobolev spaces and fractional Laplacians. We also define the spaces of rough coefficients (Sobolev multipliers) and discuss some of the basic properties. In section \ref{section-main-sigular} we prove theorem \ref{main-theorem-singular} in detail. Finally, in section \ref{section-main-bounded} we prove theorem \ref{main-theorem-bounded} but as the proofs of both theorems are very similar we do not repeat all identical steps and we keep our focus in the differences of the proofs.

\subsection*{Acknowledgements} G.C. was partially supported by the European Research Council under Horizon 2020 (ERC CoG 770924). K.M. and J.R. were supported by Academy of Finland (Centre of Excellence in Inverse Modelling and Imaging, grant numbers 284715 and 309963). G.U. was partly supported by NSF, a Walker Family Endowed Professorship at UW and a Si Yuan Professorship at IAS, HKUST. G.C. would like to thank Angkana R\"uland for helpful discussions and her hospitality during his visit to Max Planck Institute for Mathematics in the Sciences. J.R. and G.U. wish to thank Maarten V. de Hoop, Rice University and Simons Foundation for providing the support to participate 2020 MATH + X Symposium on
Inverse Problems and Deep Learning, Mitigating Natural Hazards, where this project was initiated.

\section{Preliminaries}\label{preliminaries}\label{sec:preliminaries}

In this section we recall some basic theory of Sobolev spaces, Fourier analysis and fractional Laplacians on $\R^\dimens$. We also introduce the spaces of Sobolev multipliers and prove a few properties for them. Some auxiliary lemmas which are needed in the proofs of our main theorems are given as well. We follow the references~\cite{AB-psidos-and-singular-integrals, GSU20, ML-strongly-elliptic-systems, MS-theory-of-sobolev-multipliers, TRI-interpolation-function-spaces, WO-pseudodifferential-operatros} (see also section 2 in \cite{CMR20}).

\subsection{Sobolev spaces}\label{subsec-sobolev}
The (inhomogeneous) fractional $L^2$-based Sobolev space of order $r\in\R$ is defined to be
\begin{equation}
H^{r}(\R^\dimens)=\{u\in\tempered(\R^\dimens): \ifourier(\langle\cdot\rangle^r\hat{u})\in L^2(\R^\dimens)\}
\end{equation}
equipped with the norm
\begin{equation}
\aabs{u}_{H^{r}(\R^\dimens)}=\aabs{\ifourier(\langle\cdot\rangle^r\hat{u})}_{L^2(\R^\dimens)}.
\end{equation}
Here $\hat{u}=\fourier (u)$ is the Fourier transform of a tempered distribution $u\in\tempered(\R^\dimens)$, $\ifourier$ is the inverse Fourier transform and $\langle x\rangle=(1+\abs{x}^2)^{1/2}$. We define the fractional Laplacian of order $s\in\R^+\setminus\Z$ as $\fraclaplace \varphi=\ifourier(\abs{\cdot}^{2s}\hat{\varphi})$ where $\varphi\in\schwartz(\R^\dimens)$ is a Schwartz function. Then $\fraclaplace$ extends to a bounded operator $\fraclaplace\colon H^r(\R^\dimens)\rightarrow H^{r-2s}(\R^\dimens)$ for all $r\in\R$ by density of $\schwartz(\R^\dimens)$ in $H
^r(\R^\dimens)$~\cite[Lemma 2.1]{GSU20} (see also~\cite[Section 2.2]{CMR20}). It would be possible to define the fractional Laplacian in many other ways, according to the intended application (check e.g.~\cite{DL-fractional-laplacians, KWA-ten-definitions-fractional-laplacian, LPGSGZMCMAK-what-is-fractional-laplacian}). In particular, our global definition of the fractional Laplacian using Fourier transform will be different from the spectral definition used in~\cite{helin} (see for example~\cite{DL-fractional-laplacians, SV-fractional-laplacians-are-different}).

Let $\Omega\subset\R^\dimens$ be an open set and $F\subset\R^\dimens$ a closed set. We define the following Sobolev spaces
\begin{align}
H_F^r(\R^\dimens) &=\{u\in H^r(\R^\dimens): \spt(u)\subset F\} \\
\widetilde{H}^r(\Omega)&= \ \text{closure of} \ C_c^{\infty}(\Omega) \ \text{in} \  H^r(\R^\dimens) \\
H^r(\Omega)&=\{u|_\Omega: u\in H^r(\R^\dimens)\} \\
H_0^r(\Omega)&= \ \text{closure of} \ C_c^{\infty}(\Omega) \ \text{in} \ H^r(\Omega). 
\end{align}
It trivially follows that $\widetilde{H}^r(\Omega)\subset H_0^r(\Omega)$ and $\widetilde{H}^r(\Omega)\subset H^r_{\overline{\Omega}}(\R^\dimens)$. Further, we have $(\widetilde{H}^r(\Omega))^*=H^{-r}(\Omega)$ and $(H^r(\Omega))^*=\widetilde{H}^{-r}(\Omega)$ for any open set $\Omega$ and $r\in\R$~\cite[Theorem 3.3]{CWHM-sobolev-spaces-on-non-lipchtiz-domains}. If~$\Omega$ is in addition a Lipschitz domain, then we have $\widetilde{H}^r(\Omega)=H_{\overline{\Omega}}^r(\R^\dimens)$ for all $r\in\R$ and $H_0^r(\Omega)=H^r_{\overline{\Omega}}(\R^\dimens)$ when $r\geq 0$ such that $r\notin \{\frac{1}{2}, \frac{3}{2}, \frac{5}{2} \dotso\}$~\cite[Theorems 3.29 and 3.33]{ML-strongly-elliptic-systems}.

More generally, let $1\leq p\leq \infty$ and $r\in\R$. We define the Bessel potential space
\begin{equation}
H^{r, p}(\R^\dimens)=\{u\in\tempered(\R^\dimens): \ifourier(\langle\cdot\rangle^r\hat{u})\in L^p(\R^\dimens)\}
\end{equation}
equipped with the norm
\begin{equation}
\aabs{u}_{H^{r,p}(\R^\dimens)}=\aabs{\ifourier(\langle\cdot\rangle^r\hat{u})}_{L^p(\R^\dimens)}.
\end{equation}
We also write $\ifourier (\langle\cdot\rangle^r\hat{u})=:J^r u$ where the Fourier multiplier $J=(\id-\Delta)^{1/2}$ is called the Bessel potential. We have the continuous inclusions $H^{r, p}(\R^\dimens)\hookrightarrow H^{t, p}(\R^\dimens)$ whenever $r\geq t$~\cite[Theorem 6.2.3]{BL-interpolation-spaces}. By the Mikhlin multiplier theorem one can show that $\fraclaplace\colon H^{r, p}(\R^\dimens)\rightarrow H
^{r-2s, p}(\R^\dimens)$ is continuous whenever $s\geq 0$ and $1<p<\infty$ (see~\cite[Remark 2.2]{GSU20} and~\cite[Theorem 7.2]{AB-psidos-and-singular-integrals}). 
The local version of the space $H^{r, p}(\R^\dimens)$ is defined as earlier by the restrictions
\begin{equation}
H^{r, p}(\Omega)=\{u|_\Omega: u\in H^{r, p}(\R^\dimens)\}
\end{equation}
where $\Omega\subset\R^\dimens$ is any open set. This space is equipped with the quotient norm
\begin{equation}
\aabs{v}_{H^{r, p}(\Omega)}=\inf \{\aabs{w}_{H^{r, p}(\R^\dimens)}: w\in H^{r, p}(\R^\dimens), \  w|_\Omega=v\}.
\end{equation}
We have the continuous inclusions $H^{r, p}(\Omega)\hookrightarrow H^{t, p}(\Omega)$ whenever $r\geq t$ by the definition of the quotient norm. 

We also define the spaces
\begin{align}
H_F^{r, p}(\R^\dimens) &=\{u\in H^{r, p}(\R^\dimens): \spt(u)\subset F\} \\
\widetilde{H}^{r, p}(\Omega)&= \ \text{closure of} \ C_c^{\infty}(\Omega) \ \text{in} \  H^{r, p}(\R^\dimens) \\
H_0^{r, p}(\Omega)&= \ \text{closure of} \ C_c^{\infty}(\Omega) \ \text{in} \ H^{r, p}(\Omega)
\end{align}
where $F\subset\R^\dimens$ is a closed set. Note that $\widetilde{H}^{r, p}(\Omega)\subset H_0^{r, p}(\Omega)$ since the restriction map $|_\Omega\colon H^{r, p}(\R^\dimens)\rightarrow H^{r, p}(\Omega)$ is by definition continuous. One can also see that $\widetilde{H}^{r, p}(\Omega)\subset H^{r, p}_{\overline{\Omega}}(\R^\dimens).$ If $\Omega$ is a bounded $C^\infty$-domain and $1<p<\infty$, then we have \cite[Theorem 1 in section 4.3.2]{TRI-interpolation-function-spaces}
\begin{align}
\widetilde{H}^{r, p}(\Omega)&=H^{r, p}_{\overline{\Omega}}(\R^\dimens), \quad r\in\R \\
H^{r, p}_0(\Omega)&=H^{r, p}(\Omega), \quad r\leq \frac{1}{p}.
\end{align}

Some authors (especially in~\cite{CLR18, RS-fractional-calderon-low-regularity-stability}) use the notation $W^{r, p}(\Omega)$ for Bessel potential spaces. We have decided to use the notation $H^{r, p}(\Omega)$ so that these spaces are not confused with the Sobolev-Slobodeckij spaces which are in general different from the Bessel potential spaces~\cite[Remark 3.5]{DINEPV-hitchhiker-sobolev}.

The equation \eqref{eq:fractionalpdo} we study is nonlocal. Instead of putting boundary conditions we impose exterior values for the equation. This can be done by saying that $u=f$ in $\Omega_e$ if $u-f\in\widetilde{H}^s(\Omega)$. Motivated by this we define the (abstract) trace space $X=H^r(\R^\dimens)/\widetilde{H}^r(\Omega)$, i.e. functions in $X$ are the same (have the same trace) if they agree in $\Omega_e$. If~$\Omega$ is a Lipschitz domain, then we have $X=H^r(\Omega_e)$ and $X^*=H^{-r}_{\overline{\Omega}_e}(\R^\dimens)$~\cite[p.463]{GSU20}.

\subsection{Properties of the fractional Laplacian}\label{subsec-laplacian}
The fractional Laplacian admits two important properties which we need in our proofs. The first one is unique continuation property (UCP) which is used in proving the Runge approximation property.
\begin{lemma}[UCP]
\label{lemma:ucpoffractionallaplacian}
Let $s\in\R^+\setminus\Z$, $r\in\R$ and $u\in H^{r}(\R^\dimens)$. If $(-\Delta)^s u|_V=0$ and $u|_V=0$ for some nonempty open set $V\subset\R^\dimens$, then $u=0$.
\end{lemma}
Lemma \ref{lemma:ucpoffractionallaplacian} is proved in~\cite{CMR20} for $s>1$ by reducing the problem to the UCP result for $s \in (0,1)$ in \cite{GSU20}. Note that such property is not true for local operators like the classical Laplacian $(-\Delta)$. The second property we need is the Poincar\'e inequality, which is used in showing that the forward problem for the perturbed fractional Schr\"odinger equation is well-posed.

\begin{lemma}[Poincar\'e inequality]
\label{lemma:fractionalpoincareinequliaty}
Let $s\in\R^+\setminus\Z$, $K\subset\R^\dimens$ compact set and $u\in H^s_K(\R^\dimens)$. There exists a constant $c=c(n, K, s)> 0$ such that
\begin{equation}
\aabs{u}_{L^2(\R^\dimens)}\leq c\aabs{(-\Delta)^{s/2}u}_{L^2(\R^\dimens)}.
\end{equation}
\end{lemma}
Many different proofs for lemma \ref{lemma:fractionalpoincareinequliaty} are given in~\cite{CMR20}. We note that in the literature the fractional Poincar\'e inequality is typically considered only when $s \in (0,1)$. 

Finally, we recall the fractional Leibniz rule, also known as the Kato-Ponce inequality. It is used to show the boundedness of the bilinear forms associated to the perturbed fractional Schr\"odinger equation in the case when the coefficients of the PDO have bounded fractional derivatives.
\begin{lemma}[Kato-Ponce inequality]
\label{lemma:katoponce}
Let $s\geq 0$, $1<r<\infty$, $1<q_1\leq\infty$ and $1<p_2\leq\infty$ such that $\frac{1}{r}=\frac{1}{p_1}+\frac{1}{q_1}=\frac{1}{p_2}+\frac{1}{q_2}$. If $f\in L^{p_2}(\R^\dimens)$, $J^s f\in L^{p_1}(\R^\dimens)$, $g\in L^{q_1}(\R^\dimens)$ and $J^s g\in L^{q_2}(\R^\dimens)$, then $J^s(fg)\in L^r(\R^\dimens)$ and
\begin{equation}
\aabs{J^s(fg)}_{L^r(\R^\dimens)}\leq C(\aabs{J^sf}_{L^{p_1}(\R^\dimens)}\aabs{g}_{L^{q_1}(\R^\dimens)}+\aabs{f}_{L^{p_2}(\R^\dimens)}\aabs{J^sg}_{L^{q_2}(\R^\dimens)})
\end{equation}
where $J^s$ is the Bessel potential of order $s$ and $C=C(s, \dimens, r, p_1, p_2, q_1, q_2)$.
\end{lemma}
The proof of lemma \ref{lemma:katoponce} can be found in~\cite{GK-schrodinger-operators} (see also~\cite{GO14, KAPO-commutator-estimates-kato-ponce}).

\subsection{Spaces of rough coefficients}\label{subsec-rough}

Following \cite[Ch. 3]{MS-theory-of-sobolev-multipliers}, we introduce the space of multipliers $M(H^r\rightarrow H^t)$ between pairs of Sobolev spaces. Here we are assuming that $r,t\in \mathbb R$. The coefficients of $P(x,D)$ in theorem \ref{main-theorem-singular} will be picked from such spaces of multipliers.

If $f\in \mathcal D'(\R^n)$ is a distribution, we say that $f\in M(H^r\rightarrow H^t)$ whenever the norm 
$$\|f\|_{r,t} := \sup \{\abs{\ip{f}{u v}} \,;\, u,v \in C_c^\infty(\mathbb R^n), \norm{u}_{H^r(\R^n)} = \norm{v}_{H^{-t}(\R^n)} =1 \}$$
is finite. Here $uv$ indicates the pointwise product of functions, while $\ip{\cdot}{\cdot}$ is the duality pairing. If the distribution $f$ happens to be a function, the duality pairing can be defined as $$ \langle f, uv \rangle = \int_{\mathbb R^n} f(x)u(x)v(x) \der x. $$  By $M_0(H^r\rightarrow H^t)$ we indicate the closure of $C^\infty_c(\mathbb R^n)$ in $M(H^r\rightarrow H^t)\subset \mathcal D'(\mathbb R^n)$. If $f\in M(H^r\rightarrow H^t)$ and $u,v \in C_c^\infty(\mathbb R^n)$ are both non-vanishing, we have the multiplier inequality
\begin{equation}\label{multiplier-inequality}
    \abs{\ip{f}{uv}} = \left|\ip{ f}{\frac{u}{\norm{u}_{H^r(\R^n)}}\frac{v}{\norm{v}_{H^{-t}(\R^n)}}}\right| \norm{u}_{H^r(\R^n)} \norm{v}_{H^{-t}(\R^n)} \leq \|f\|_{r,t}\norm{u}_{H^r(\R^n)} \norm{v}_{H^{-t}(\R^n)}.
\end{equation}

By the density of $C_c^\infty(\R^n) \times C_c^\infty(\R^n)$ in $ H^r(\R^n) \times H^{-t}(\R^n)$ with respect to the product norm $\norm{(u,v)} = \max\{\norm{u}_{H^r(\R^n)},\norm{v}_{H^{-t}(\R^n)}\}$ and estimate \eqref{multiplier-inequality}, there is a unique continuous extension of $(u,v) \mapsto \langle f, uv \rangle$ acting on $(u,v)\in H^r(\R^n)\times H^{-t}(\R^n)$. More precisely, each $f \in M(H^r\rightarrow H^t)$ gives rise to a linear multiplication map $m_f : H^r(\R^n) \rightarrow H^t(\R^n)$ defined by 
\begin{align}
    \langle m_f(u),v \rangle := \lim_{i \to \infty}\langle f,u_iv_i \rangle \quad \mbox{for all} \quad (u,v)\in H^r(\R^n)\times H^{-t}(\R^n),
\end{align}
where $(u_i,v_i) \in C_c^\infty(\R^n) \times C_c^\infty(\R^n)$ is any Cauchy sequence in $H^r(\R^n) \times H^{-t}(\R^n)$ converging to $(u,v)$. The existence of the limit is granted by completeness and formula \eqref{multiplier-inequality}, which ensures that $\langle f,u_iv_i \rangle$ is also a Cauchy sequence. In fact, we have
\begin{equation*}
\begin{split}
    \abs{\langle f,u_m v_m \rangle-\langle f,u_n v_n \rangle }
    &\leq \abs{\langle f,u_m(v_m-v_n) \rangle} + \abs{\langle f,(u_m-u_n)v_n \rangle} \\
    &\leq \norm{f}_{r,t}\left(\norm{u_m}_{H^r(\R^n)}\norm{v_m-v_n}_{H^{-t}(\R^n)}+\norm{u_m-u_n}_{H^r(\R^n)}\norm{v_n}_{H^{-t}(\R^n)}\right)\\
    &\leq \norm{f}_{r,t}\left(\norm{u_m}_{H^r(\R^n)}+\norm{v_n}_{H^{-t}(\R^n)}\right)\norm{(u_m,v_m)-(u_n,v_n)},
\end{split}
\end{equation*}
where $\norm{u_m}_{H^r(\R^n)}+\norm{v_n}_{H^{-t}(\R^n)}$ is bounded by a constant independent of $m$ and $n$. The independence of the limit on the particular sequence $(u_i,v_i)$ can be showed by a similar estimate. 

We can analogously define the unique adjoint multiplication map $m_f^*: H^{-t}(\R^n) \to H^{-r}(\R^n)$ such that 
\[\ip{m_f^*(v)}{u} := \lim_{i \to \infty}\langle f,u_iv_i \rangle \quad \mbox{for all} \quad (u,v)\in H^r(\R^n)\times H^{-t}(\R^n).\] Since one sees that the adjoint of $m_f$ is $m_f^*$, the chosen notation is justified. For convenience, in the rest of the paper we will just write $fu$ for both $m_f(u)$ and $m_f^*(u)$.

\begin{remark}
The spaces of rough coefficients we use are generalizations of the ones considered in \cite{RS-fractional-calderon-low-regularity-stability}. In fact, the space $Z^{-s}(\mathbb R^n)$ used there coincides with our space $M(H^s\rightarrow H^{-s})$. 
\end{remark}

In the next lemma we state some elementary properties of the spaces of multipliers. Other interesting properties may be found in \cite{MS-theory-of-sobolev-multipliers}.

\begin{lemma}\label{lemma-properties-of-multipliers}
Let $\lambda, \mu \geq 0$ and $r,t\in\mathbb R$. Then
 \begin{enumerate}[(i)]
    \item\label{multiplier-lemma-item1} $M(H^r\rightarrow H^t) = M(H^{-t}\rightarrow H^{-r})$, and the norms associated to the two spaces also coincide.
    \medskip
    \item\label{multiplier-lemma-item2} $M(H^{r-\lambda}\rightarrow H^{t+\mu}) \hookrightarrow M(H^r\rightarrow H^t)$ continuously.
    \medskip
    \item\label{multiplier-lemma-item3} $M(H^r\rightarrow H^t) = \{0\}$ whenever $r<t$.
\end{enumerate}
\end{lemma}

\begin{proof}\ref{multiplier-lemma-item1} Let $f\in \mathcal D'(\R^\dimens)$ be a distribution. Then by just using the definition we see that
    \begin{align}
        \|f\|_{r,t} & = \sup \{\abs{\ip{f}{u v}} \,;\, u,v \in C_c^\infty(\mathbb R^n), \norm{u}_{H^r(\R^n)} = \norm{v}_{H^{-t}(\R^n)} =1 \} \\ 
        & = \sup \{\abs{\ip{f}{v u}} \,;\, v,u \in C_c^\infty(\mathbb R^n), \norm{v}_{H^{-t}(\R^n)} = \norm{u}_{H^{-(-r)}(\R^n)} =1 \} = \|f\|_{-t,-r}.
    \end{align}
    
\noindent \ref{multiplier-lemma-item2} Observe that the given definition of $\|f\|_{r,t}$ is equivalent to the following:
\begin{align}
    \|f\|_{r,t} = \sup \{\abs{\ip{f}{u v}} \,;\, u,v \in C_c^\infty(\mathbb R^n), \norm{u}_{H^r(\R^n)} \leq 1, \norm{v}_{H^{-t}(\R^n)} \leq 1 \}.
\end{align}
\noindent Since $\lambda, \mu \geq 0$, we also have $$ \norm{u}_{H^{r-\lambda}(\R^n)} \leq \norm{u}_{H^r(\R^n)}, \quad \norm{v}_{H^{-(t+\mu)}(\R^n)} \leq \norm{v}_{H^{-t}(\R^n)}. $$
\noindent This implies $ \|f\|_{r,t} \leq \|f\|_{r-\lambda,t+\mu} $, which in turn gives the wanted inclusion.

\ref{multiplier-lemma-item3} If $0\leq r < t$, then this was considered in \cite[Ch. 3]{MS-theory-of-sobolev-multipliers}. The proof given there recalls the easier one for Sobolev spaces (\cite[Sec. 2.1]{MS-theory-of-sobolev-multipliers}), which is based on the explicit computation of derivatives of aptly chosen exponential functions.

If $r < t \leq 0$, then by point \ref{multiplier-lemma-item1} we have $M(H^r\rightarrow H^t) = M(H^{-t}\rightarrow H^{-r})$. We need to show that $M(H^{-t}\rightarrow H^{-r}) = \{0\}$ whenever $0\leq -t<-r$. This reduces the problem back to the case of non-negative Sobolev scales.

If $r \leq 0 < t$, then $-r \geq 0$. Now by point \ref{multiplier-lemma-item2}, we have $M(H^r\rightarrow H^t) \subseteq M(H^{r+(-r)}\rightarrow H^t) = M(L^2\rightarrow H^t) $. It is therefore enough to show that this last space is trivial, which again immediately follows from the case of non-negative Sobolev scales.

If $r < 0 \leq t$, then the problem can be reduced again to the earlier cases.
\end{proof}

\begin{remark}
We also have $M_0(H^{r-\lambda}\rightarrow H^{t+\mu}) \subseteq M_0(H^r\rightarrow H^t)$ whenever $\lambda, \mu\geq 0$, since the inclusion in~\ref{multiplier-lemma-item2} is continuous.
\end{remark}

\begin{remark}\label{why-exclude-higher-order-PDOs}
In light of lemma \ref{lemma-properties-of-multipliers}~\ref{multiplier-lemma-item2} we are only interested in $M(H^r\rightarrow H^t)$ in the case $r\geq t$, the case $r<t$ being trivial. For our theorem \ref{main-theorem-singular}, this translates into the condition $m \leq 2s$. We decided not to consider the limit case $m=2s$ in this work, as our machinery (in particular, the coercivity estimate \eqref{coercivity-3}) breaks down in this case. However, it should be noted that since by assumption we have $m\in\mathbb Z$ and $s\not\in\mathbb Z$, the equality $m=2s$ can only arise if $m$ is odd, which forces $s=1/2+k$ with $k\in \mathbb Z$. This case was excluded in \cite{CLR18, GSU20} as well.
\end{remark}

Propositions~\ref{lemma:generalcoefficients} and~\ref{lemma:lowerordercoefficients} relate our spaces of multipliers with some special Bessel potential spaces. This is interesting since in the coming section \ref{section-main-sigular} we will consider the inverse problem for coefficients coming from such spaces. We now prove those propositions.

\begin{proof}[Proof of proposition \ref{lemma:generalcoefficients}]
Throughout the proof we assume that $u, v\in C_c^\infty(\R^\dimens)$ such that $\aabs{u}_{H^{-r}(\R^\dimens)}=\aabs{v}_{H^t(\R^\dimens)}=1$. In parts \ref{item1} and \ref{item2} we can assume that $r'<t$ since if $r'\geq t$, then we have the continuous inclusion $H^{r', \infty}(\Omega)\hookrightarrow H^{r'', \infty}(\Omega)$ where $\max\{0, r\}\leq r''<t$ (such $r''$ always exists since $t>\max\{0, r\}$).

\ref{item1} Let $f\in\widetilde{H}^{r', \infty}(\Omega)$. Now $f=f_1+f_2$ where $f_1\in C^\infty_c(\Omega)$ and $\aabs{f_2}_{H^{r', \infty}(\R^\dimens)}\leq\epsilon$. Then
\begin{align}
\abs{\ip{f_2}{uv}}\leq\aabs{f_2 v}_{H^{r'}(\R^\dimens)}\aabs{u}_{H^{-r'}(\R^\dimens)}&\leq C\aabs{f_2}_{H^{r', \infty}(\R^\dimens)}\aabs{v}_{H^{r'}(\R^\dimens)}\aabs{u}_{H^{-r}(\R^\dimens)} \\
&\leq C\epsilon\aabs{v}_{H^t(\R^\dimens)}=C\epsilon.
\end{align}
Here we used the Kato-Ponce inequality (lemma \ref{lemma:katoponce})
\begin{align}
\aabs{J^{r'}(f_2 v)}_{L^2(\R^n)} & \leq C (\aabs{f_2}_{L^\infty(\R^n)} \aabs{J^{r'} v}_{L^2(\R^n)} + \aabs{J^{r'} f_2}_{L^\infty(\R^n)} \aabs{v}_{L^2(\R^n)}) \\
& \leq C \aabs{f_2}_{H^{r',\infty}(\R^n)} \aabs{v}_{H^{r'}(\R^n)}
\end{align}
and the assumption $\max\{0, r\}\leq r'<t$. Therefore $\aabs{f-f_1}_{-r, -t}=\aabs{f_2}_{-r, -t}\leq C\epsilon$ which shows that $f\in M_0(H^{-r}\rightarrow H^{-t})$.

\ref{item2} Let $f\in H^{r', \infty}_0(\Omega)$. Now $f=f_1+f_2$ where $f_1\in C^\infty_c(\Omega)$ and $\aabs{f_2}_{H^{r', \infty}(\Omega)}\leq\epsilon$. By the definition of the quotient norm $\aabs{\cdot}_{H^{r', \infty}(\Omega)}$ we can take $F\in H^{r', \infty}(\R^\dimens)$ such that $F|_\Omega=f_2$ and $\aabs{F}_{H^{r', \infty}(\R^\dimens)}\leq 2\aabs{f_2}_{H^{r', \infty}(\Omega)}$. The assumptions imply the duality $(H^{-r'}(\Omega))^*=H^{r'}_0(\Omega)\subset H
^{r'}(\Omega)$. Using the Kato-Ponce inequality for the extension $F$ we obtain as in the proof of part \ref{item1} that
\begin{align}
\aabs{J^{r'}(F v)}_{L^2(\R^n)}&\leq C \aabs{F}_{H^{r',\infty}(\R^n)} \aabs{v}_{H^{r'}(\R^n)}\leq 2C\aabs{f_2}_{H^{r', \infty}(\Omega)}\aabs{v}_{H^t(\R^n)}\leq 2C\epsilon
\end{align}
and hence
\begin{align}
\abs{\ip{f_2}{uv}}&\leq\aabs{f_2 v}_{(H^{-r'}(\Omega))^*}\aabs{u}_{H^{-r'}(\Omega)}\leq\aabs{f_2 v}_{H^{r'}(\Omega)}\aabs{u}_{H^{-r}(\R^\dimens)} \\&\leq \aabs{J^{r'}(F v)}_{L^2(\R^n)}\leq 2C\epsilon.
\end{align}
This shows that $f\in M_0(H^{-r}\rightarrow H^{-t})$.

\ref{item3} Let $f\in\widetilde{H}^{r'}(\Omega)$. Now $f=f_1+f_2$ where $f_1\in C^\infty_c(\Omega)$ and $\aabs{f_2}_{H^{r'}(\R^\dimens)}\leq\epsilon$. Now~\cite[Theorem 7.3]{BH2017} implies the continuity of the multiplication $H^{r'}(\R^\dimens)\times H^t(\R^\dimens)\hookrightarrow H^t(\R^\dimens)$ when $r'\geq t$ and $r'>n/2$. We obtain 
\begin{align}
\abs{\ip{f_2}{uv}}\leq\aabs{f_2 v}_{H^{t}(\R^\dimens)}\aabs{u}_{H^{-t}(\R^\dimens)}&\leq C\aabs{f_2}_{H^{r'}(\R^\dimens)}\aabs{v}_{H^t(\R^\dimens)}\aabs{u}_{H^{-r}(\R^\dimens)} \leq C\epsilon.
\end{align}
Hence $f\in M_0(H^{-r}\rightarrow H^{-t})$. If $\Omega$ is a Lipschitz domain, then $H^{r'}_{\overline{\Omega}}(\R^\dimens)=\widetilde{H}^{r'}(\Omega)$. If in addition $r'\notin\{\frac{1}{2}, \frac{3}{2}, \frac{5}{2}, \dotso\}$, we also have $H^{r'}_0(\Omega)=\widetilde{H}^{r'}(\Omega)$.
\end{proof}

\begin{proof}[Proof of proposition \ref{lemma:lowerordercoefficients}]
Throughout the proof we assume that $u, v\in C_c^\infty(\R^\dimens)$ such that $\aabs{u}_{L^2(\R^\dimens)}=\aabs{v}_{H^t(\R^\dimens)}=1$.

\ref{item4} Let $f\in L^p(\Omega)$. By density of $C^\infty_c(\Omega)$ in $L^p(\Omega)$ we have $f=f_1+f_2$ where $f_1\in C^\infty_c(\Omega)$ and $\aabs{\widetilde{f}_2}_{L^p(\R^\dimens)}\leq\epsilon$ where $\widetilde{f}_2$ is the zero extension of $f_2\in L^p(\Omega)$. The assumptions on $p$ imply the continuity of the multiplication $L^p(\R^\dimens)\times H^t(\R^\dimens)\hookrightarrow L^2(\R^\dimens)$ (\cite[Theorem 7.3]{BH2017}) and we have 
\begin{align}
\abs{\ip{\widetilde{f}_2}{uv}}\leq\aabs{\widetilde{f}_2 v}_{L^2(\R^\dimens)}\aabs{u}_{L^2(\R^\dimens)}&\leq C\aabs{\widetilde{f}_2}_{L^p(\R^\dimens)}\aabs{v}_{H^t(\R^\dimens)}\leq C\epsilon.
\end{align}
This gives that $f\in M_0(H^0\rightarrow H^{-t})$. If $\Omega$ is bounded, we have $L^\infty(\Omega)\hookrightarrow L^p(\Omega)$ for all $1\leq p<\infty$, giving the second claim.

\ref{item5} Let $f\in \widetilde{H}^r(\Omega)$. Now we have $f=f_1+f_2$ where $f_1\in C^\infty_c(\Omega)$ and $\aabs{f_2}_{H^r(\R^\dimens)}\leq\epsilon$. The assumptions on $r$ imply that the multiplication $H^r(\R^\dimens)\times H^t(\R^\dimens)\hookrightarrow L^2(\R^\dimens)$ is continuous (\cite[Theorem 7.3]{BH2017}). We obtain
\begin{align}
\abs{\ip{f_2}{uv}}\leq\aabs{f_2 v}_{L^2(\R^\dimens)}\aabs{u}_{L^2(\R^\dimens)}&\leq C\aabs{f_2}_{H^r(\R^\dimens)}\aabs{v}_{H^t(\R^\dimens)}\leq C\epsilon
\end{align}
and therefore $f\in M_0(H^0\rightarrow H^{-t})$. The claims for $H^r_{\overline{\Omega}}(\R^\dimens)$ and $H^r_0(\Omega)$ follow as in the proof of part \ref{item3} of lemma \ref{lemma:generalcoefficients} from the usual identifications for Lipschtiz domains.
\end{proof}

\section{Main theorem for singular coefficients}\label{section-main-sigular}
In this section, to shorten the notation, we will write $\aabs{\cdot}_{H^s}$, $\aabs{\cdot}_{L^2}$ and so on for the global norms in $\R^\dimens$ when the base set is not written explicitly.

\subsection{Well-posedness of the forward problem}\label{subsec-WP-singular}

Consider the problem
\begin{align}\label{problem}
(-\Delta)^s u + \sum_{|\alpha|\leq m} a_\alpha(D^\alpha u) & = F \quad \mbox{in} \;\;\Omega, \\
u & = f  \quad \mbox{in} \;\;\Omega_e 
\end{align}
and the corresponding \emph{adjoint-problem}
\begin{align}\label{adjoint-problem}
(-\Delta)^s u^* + \sum_{|\alpha|\leq m} (-1)^{|\alpha|} D^\alpha (a_\alpha u^*) & = F^* \quad \mbox{in} \;\;\Omega, \\
u^* & = f^*  \quad \mbox{in} \;\;\Omega_e .
\end{align}
Note that if $u, u^*\in H^s(\R^\dimens)$ and $a_\alpha\in M(H^{s-\abs{\alpha}}\rightarrow H^{-s})=M(H^s\rightarrow H^{\abs{\alpha}-s})$, then $a_\alpha(D^\alpha u)\in H^{-s}(\R^\dimens)$ and $D^\alpha(a_\alpha u^*)\in H^{-s}(\R^\dimens)$ matching with $\fraclaplace u, \fraclaplace u^*\in H^{-s}(\R^\dimens)$.

The problems \eqref{problem} and \eqref{adjoint-problem} are associated to the bilinear forms 
\begin{align}\label{bilinear}
B_P(v,w) := \langle (-\Delta)^{s/2}v, (-\Delta)^{s/2}w \rangle + \sum_{|\alpha|\leq m} \langle a_\alpha ,(D^\alpha v) w \rangle
\end{align}
and
\begin{align}\label{adjoint-bilinear}
B^*_P(v,w) := \langle (-\Delta)^{s/2}v, (-\Delta)^{s/2}w \rangle + \sum_{|\alpha|\leq m} \langle a_\alpha, v(D^\alpha w) \rangle,
\end{align}
defined on $v,w \in C^\infty_c(\R^\dimens)$.

\begin{remark}
Observe that $B_P$ is not symmetric, which motivates the introduction of the bilinear form $B_P^*$. Moreover, one sees by simple inspection that $B_P(v,w) = B_P^*(w,v)$ for all $v,w \in C^\infty_c(\R^\dimens)$. This identity holds for $v,w \in H^s(\R^\dimens)$ as well by density, thanks to the following boundedness lemma.
\end{remark}

\begin{lemma}[Boundedness of the bilinear forms]\label{boundedness-bilinear}
Let $s \in \mathbb R^+ \setminus \mathbb Z$ and $m\in \mathbb N$ such that $2s \geq m$, and let $a_{\alpha} \in M(H^{s-|\alpha|}\rightarrow H^{-s})$. Then $B_P$ and $B_P^*$ extend as bounded bilinear forms on $H^s(\mathbb R^n)\times H^s(\mathbb R^n)$.
\end{lemma}

\begin{proof}
We only prove the boundedness of $B_P$, as for $B_P^*$ one can proceed in the same way. The proof is a simple calculation following from inequality \eqref{multiplier-inequality}. Let $u, v \in C_c^\infty(\R^n)$. We can then estimate that
\begin{align}
|B_P(v,w)| & \leq |\langle (-\Delta)^{s/2}v, (-\Delta)^{s/2}w \rangle | + \sum_{|\alpha|\leq m} |\langle a_\alpha , (D^\alpha v) w \rangle| \\
& \leq \|w\|_{H^s(\mathbb R^n)} \|v\|_{H^s(\mathbb R^n)} + \sum_{|\alpha|\leq m}  \|a_\alpha\|_{s-|\alpha|,-s} \|D^\alpha v\|_{H^{s-|\alpha|}(\mathbb R^n)}\|w\|_{H^s(\mathbb R^n)} \\ & 
\leq \left( 1+ \sum_{|\alpha|\leq m}  \|a_\alpha\|_{s-|\alpha|,-s}\right) \|w\|_{H^s(\mathbb R^n)} \|v\|_{H^s(\mathbb R^n)}.
\end{align}
Now the claim follows from the density of $C_c^\infty(\R^n)$ in $H^s(\R^n)$.
\end{proof}

Next we shall define the concept of weak solution to problems \eqref{problem} and \eqref{adjoint-problem}:
\begin{definition}[Weak solutions]
Let $f,f^* \in H^s(\mathbb R^n)$ and $F, F^* \in (\widetilde{H}^{s}(\Omega))^*$. We say that $u\in H^s(\mathbb R^n)$ is a weak solution to \eqref{problem} when $u-f \in \widetilde H^s(\Omega)$ and $B_P(u,v)=F(v)$ for all $v\in \widetilde H^s(\Omega)$. Similarly, we say that $u^*\in H^s(\mathbb R^n)$ is a weak solution to \eqref{adjoint-problem} when $u^*-f^* \in \widetilde H^s(\Omega)$ and $B^*_P(u^*,v)=F^*(v)$ for all $v\in \widetilde H^s(\Omega)$.
\end{definition}

In order to prove the existence and uniqueness of weak solutions, we use the following form of Young's inequality, which holds for all $a,b,\eta \in \mathbb R^+$ and $p,q \in (1,\infty)$ such that $1/p+1/q=1$: 
\begin{align}\label{young-general}
    ab \leq \frac{(q\eta)^{-p/q}}{p}a^p + \eta b^q.
\end{align}
The validity of \eqref{young-general} is easily proved by choosing $a_1=a(q\eta)^{-1/q}$ and $b_1 = b(q\eta)^{1/q}$ in Young's inequality  $a_1b_1 \leq a_1^p/p + b_1^q/q$. 

\begin{lemma}[Well-posedness]\label{well-posedness}
Let $\Omega \subset \mathbb R^n$ be a bounded open set. Let $s \in \mathbb R^+ \setminus \mathbb Z$ and $m\in \mathbb N$ be such that $2s > m$, and let $a_{\alpha} \in M_0(H^{s-|\alpha|}\rightarrow H^{-s})$. There exist a real number $\mu > 0$ and a countable set $\Sigma \subset (-\mu, \infty)$ of eigenvalues $\lambda_1 \leq \lambda_2 \leq ... \rightarrow \infty$ such that if $\lambda \in \mathbb R \setminus \Sigma$, for any $f\in H^s(\mathbb R^n)$ and $F\in (\widetilde{H}
^s(\Omega))^*$ there exists a unique $u\in H^s(\mathbb R^n)$ such that $u-f \in \widetilde H^s(\Omega)$ and
$$ B_P(u,v)-\lambda \langle u,v \rangle = F(v) \quad \mbox{for all} \quad v \in \widetilde H^s(\Omega).$$
One has the estimate 
$$ \|u\|_{H^s(\mathbb R^n)} \leq C\left( \|f\|_{H^s(\mathbb R^n)} + \|F\|_{(\widetilde{H}^s(\Omega))^*} \right). $$
The function $u$ is also the unique $u\in H^s(\mathbb R^n)$ satisfying $$r_\Omega \left( (-\Delta)^s + \sum_{|\alpha|\leq m} a_\alpha D^\alpha  -\lambda \right)u=F$$ in the sense of distributions in $\Omega$ and $u-f \in \widetilde H^s(\Omega)$. Moreover, if \eqref{eigenvalue-condition-1} holds then $0\notin \Sigma$.
\end{lemma}

\begin{proof}
Let $\tilde u := u-f$. The above problem is reduced to finding a unique $\tilde u \in \widetilde H^s(\Omega)$ such that $B_P(\tilde u, v) - \lambda \langle \tilde u,v \rangle = \tilde F(v)$, where $\tilde F := F - B_P(f,\cdot) + \lambda\langle f,\cdot\rangle$. Observe that the modified functional $\tilde F$ belongs to $(\widetilde H^s(\Omega))^*$ as well, since by lemma \ref{boundedness-bilinear} we have for all $v\in \widetilde H^s(\Omega)$
$$|\tilde F(v)| \leq |F(v)| + |B_P(f,v)| + |\lambda| \,|\langle f,v\rangle| \leq (\aabs{F}_{(\widetilde{H}^s(\Omega))^*}+(C+|\lambda|)\|f\|_{H^s(\mathbb R^n)})\|v\|_{H^s(\mathbb R^n)}.$$

Since $a_{\alpha} \in M_0(H^{s-|\alpha|}\rightarrow H^{-s})$, for any $\epsilon > 0$ we can write $a_\alpha = a_{\alpha,1}+ a_{\alpha,2}$, where $a_{\alpha,1} \in C^\infty_c(\mathbb R^n) \cap M(H^{s-\abs{\alpha}} \to H^{-s})$ and $\|a_{\alpha,2}\|_{s-|\alpha|,-s}< \epsilon$. Thus by formula \eqref{multiplier-inequality}, the continuity of the multiplication $H^r(\R^\dimens)\times H^s(\R^\dimens)\hookrightarrow H
^s(\R^\dimens)$ for large enough $r\in\R$ (see \cite[Theorem 7.3]{BH2017}) and the fact that $a_{\alpha, 1}\in C^\infty_c(\R^\dimens)\subset H^r(\R^\dimens)$ for all $r\in\R$ we obtain

\begin{align}\label{rough-WP-1}
    |\langle a_\alpha, (D^\alpha v) w \rangle| & \leq |\langle a_{\alpha,1},(D^\alpha v) w \rangle| +  |\langle a_{\alpha,2}, (D^\alpha v) w \rangle| \\ 
    & \leq \|a_{\alpha,1}\|_{H^r(\R^n)}\|D^\alpha v\|_{H^{-s}(\R^n)} \|w\|_{H^s(\R^n)} +  \|a_{\alpha,2}\|_{s-|\alpha|,-s} \|D^\alpha v\|_{H^{s-|\alpha|}(\R^n)} \|w\|_{H^s(\R^n)} \\
    & \leq c \|w\|_{H^s(\R^n)} \left( \|a_{\alpha,1}\|_{H^r(\R^n)}\|v\|_{H^{|\alpha|-s}(\R^n)} +  \epsilon \|v\|_{H^s(\R^n)} \right)
\end{align}
where $r\in\R$ is large enough ($r>\max\{s, n/2\}$ is sufficient). If $|\alpha| < s$, from formulas \eqref{rough-WP-1} and \eqref{young-general} with $p=q=2$ we get directly 
\begin{align}\label{rough-WP-2} 
|\langle a_\alpha, (D^\alpha v) v \rangle| & \leq C\left( \|v\|_{H^s(\R^n)} \|v\|_{L^2(\R^n)} +  \epsilon \|v\|_{H^s(\R^n)}^2\right) \\ 
& \leq C (\epsilon^{-1}\|v\|_{L^2(\R^n)}^2 + \epsilon \|v\|_{H^s(\R^n)}^2 )  
\end{align}

\noindent for a constant $C$ independent of $v,w,\epsilon$. If instead $|\alpha| > s$ (observe that we can not have $|\alpha|=s$, because $s$ can not be an integer), we use the interpolation inequality
$$ \|v\|_{H^{|\alpha|-s}(\mathbb R^n)} \leq C \|v\|_{L^2(\mathbb R^n)}^{1-(|\alpha|-s)/s} \|v\|_{H^s(\mathbb R^n)}^{(|\alpha|-s)/s} = C \|v\|_{L^2(\mathbb R^n)}^{2-|\alpha|/s} \|v\|_{H^s(\mathbb R^n)}^{|\alpha|/s-1}$$

\noindent in order to get
\begin{align}
    |\langle a_\alpha, (D^\alpha v) w \rangle| & \leq C \|w\|_{H^s(\R^n)} \left( \|v\|_{L^2(\R^n)}^{2-|\alpha|/s} \|v\|_{H^s(\R^n)}^{|\alpha|/s-1}+  \epsilon \|v\|_{H^s(\R^n)} \right).
\end{align}
Then by formula \eqref{young-general} with
$$ a= \|v\|_{L^2(\R^n)}^{2-|\alpha|/s}, \quad b= \|v\|_{H^s(\R^n)}^{|\alpha|/s-1}, \quad p=\frac{s}{2s-|\alpha|}, \quad q= \frac{s}{|\alpha|-s}, \quad \eta=\epsilon$$
we obtain
\begin{align}\label{rough-WP-3}
    |\langle a_\alpha, (D^\alpha v) w \rangle| & \leq C \|w\|_{H^s(\R^n)} \left( \epsilon^{\frac{s-|\alpha|}{2s-|\alpha|}}\|v\|_{L^2(\R^n)} + \epsilon \|v\|_{H^s(\R^n)} \right)
\end{align}
\noindent for a constant $C$ independent of $v,w,\epsilon$. Now we use formula \eqref{young-general} again, but this time we choose
$$a= \|v\|_{L^2(\R^n)}, \quad b= \|v\|_{H^s(\R^n)}, \quad q=p=2, \quad \eta = \epsilon^{s/(2s-|\alpha|)}.$$ 
This leads to
\begin{align}\label{rough-WP-4}
     |\langle a_\alpha, (D^\alpha v) v \rangle| & \leq C \left( \epsilon^{\frac{s-|\alpha|}{2s-|\alpha|}}\|v\|_{L^2(\R^n)} \|v\|_{H^s(\R^n)} + \epsilon \|v\|_{H^s(\R^n)}^2\right) \\ 
     & \leq C \left( \epsilon^{\frac{-|\alpha|}{2s-|\alpha|}}\|v\|_{L^2(\R^n)}^2 + 2\epsilon \|v\|_{H^s(\R^n)}^2\right) \\ 
     & \leq C \left( \epsilon^{\frac{-|\alpha|}{2s-|\alpha|}}\|v\|_{L^2(\R^n)}^2 + \epsilon \|v\|_{H^s(\R^n)}^2\right) \\ 
     & \leq C' \left( \epsilon^{\frac{-m}{2s-m}}\|v\|_{L^2(\R^n)}^2 + \epsilon \|v\|_{H^s(\R^n)}^2\right)
\end{align}

\noindent where $C, C'$ are constants changing from line to line. Observe that $C'$ can be taken independent of $\alpha$. Eventually, using \eqref{rough-WP-2} and \eqref{rough-WP-4} we get
\begin{align}\label{coercivity-1}
B_P(v,v) & \geq \|(-\Delta)^{s/2}v\|^2_{L^2(\R^n)} - \sum_{|\alpha|\leq m} | \langle a_\alpha, (D^\alpha v) v \rangle | \\ 
& \geq  \|(-\Delta)^{s/2}v\|^2_{L^2(\R^n)} -  C' \left( (\epsilon^{\frac{-m}{2s-m}}+\epsilon^{-1})\|v\|_{L^2(\R^n)}^2 + \epsilon \|v\|_{H^s(\R^n)}^2\right).
\end{align}
\noindent By the higher order Poincar\'e inequality (lemma \ref{lemma:fractionalpoincareinequliaty}) \eqref{coercivity-1} turns into
\begin{align}
B_P(v,v) & \geq c \left(\|(-\Delta)^{s/2}v\|^2_{L^2(\R^n)}+ \|v\|_{L^2(\R^n)}^2\right)  - C' \left( (\epsilon^{\frac{-m}{2s-m}}+\epsilon^{-1})\|v\|_{L^2(\R^n)}^2 + \epsilon \|v\|_{H^s(\R^n)}^2\right) \\ 
& \geq c \|v\|_{H^s(\R^n)}^2 - C' \left( (\epsilon^{\frac{-m}{2s-m}}+\epsilon^{-1})\|v\|_{L^2(\R^n)}^2 + \epsilon \|v\|_{H^s(\R^n)}^2\right) 
\end{align}

\noindent for some constant $c=c(\Omega,n,s)$ changing from line to line. For $\epsilon$ small enough, this eventually gives the coercivity estimate
\begin{equation}\label{coercivity-3}
B_P(v,v) \geq c_0 \|v\|_{H^s(\mathbb R^n)}^2 - \mu\|v\|_{L^2(\mathbb R^n)}^2
\end{equation}
\noindent for some constants $c_0, \mu > 0$ independent of $v$.

 As a consequence of the coercivity estimate, the bilinear form $B_P(\cdot,\cdot) + \mu\langle\cdot,\cdot\rangle_{L^2(\mathbb R^n)}$ satisfies the assumptions of the Lax--Milgram theorem, and there exists a bounded linear operator $G_\mu: (\widetilde{H}^{s}(\Omega))^* \rightarrow \widetilde H^s(\Omega)$ associating each functional in $(\widetilde{H}^{s}(\Omega))^*$ to its unique representative in the bilinear form $B_P(\cdot,\cdot) + \mu\langle\cdot,\cdot\rangle_{L^2(\mathbb R^n)}$ on $\widetilde H^s(\Omega)$. Thus $\tilde u := G_\mu \tilde F$ verifies $$ B_P(\tilde u,v) + \mu\langle\tilde u,v\rangle_{L^2(\mathbb R^n)} = \tilde F(v)\quad \mbox{for all}\quad v\in \widetilde H^s(\Omega) $$
and it is the required unique solution $\tilde u \in \widetilde H^s(\Omega)$. Moreover, $G_\mu$ induces a compact operator $\tilde G_\mu : L^2(\Omega) \rightarrow L^2(\Omega)$ by the compact Sobolev embedding theorem. The remaining claims follow from the spectral theorem of compact operators for $\tilde G_\mu$ and from the Fredholm alternative as in \cite{GSU20}.
\end{proof}

By the above lemma \ref{well-posedness}, both problems \eqref{problem} and \eqref{adjoint-problem} have a countable set of Dirichlet eigenvalues. Throughout the paper we will assume that the coefficients $a_\alpha$ are such that $0$ is not a Dirichlet eigenvalue for either of the problems. That is, we assume that
\begin{align}\label{eigenvalue-condition-1}
    \Bigg\{ \begin{matrix*} \mbox{if $u \in H^s(\mathbb R^n)$ solves $(-\Delta)^s u + \sum_{|\alpha|\leq m} a_\alpha D^\alpha u = 0$ in $\Omega$ and $u|_{\Omega_e}=0$,}
    \\
    \mbox{then $u\equiv 0$}
    \end{matrix*} \Bigg. 
    \end{align}
    and 
    \begin{align}\label{eigenvalue-condition-2}
    \Bigg\{ \begin{matrix*} \mbox{if $u^* \in H^s(\mathbb R^n)$ solves $(-\Delta)^s u^* + \sum_{|\alpha|\leq m} (-1)^{|\alpha|} D^\alpha (a_\alpha  u^*) = 0$ in $\Omega$ and $u^*|_{\Omega_e}=0$,}
    \\
    \mbox{then $u^*\equiv 0$.}
    \end{matrix*} \Bigg.
\end{align}

With this in mind, we shall define the exterior DN maps associated to the problems \eqref{problem} and \eqref{adjoint-problem}. Consider the abstract trace space $X := H^s(\mathbb R^n)/\widetilde H^s(\Omega)$ equipped with the quotient norm $$\|[f]\|_X := \inf_{\phi\in\widetilde H^s(\Omega)} \|f-\phi\|_{H^s(\mathbb R^n)}, \quad f\in H^s(\mathbb R^n)$$ 
and its dual space $X^*$. 

\begin{definition}
Let $\Omega \subset \mathbb R^n$ be a bounded open set. Let $s \in \mathbb R^+ \setminus \mathbb Z$ and $m\in \mathbb N$ such that $2s> m$, and let $a_{\alpha} \in M_0(H^{s-|\alpha|}\rightarrow H^{-s})$. The exterior DN maps $\Lambda_P$ and $\Lambda_P^*$ are  $$ \Lambda_P : X\rightarrow X^* \quad \mbox{defined by}\quad \langle \Lambda_P[f],[g] \rangle := B_P(u_f, g)$$
and 
$$ \Lambda^*_P : X\rightarrow X^* \quad \mbox{defined by}\quad \langle \Lambda^*_P[f],[g] \rangle := B^*_P(u^*_f, g)$$
where $u_f, u^*_f$ are the unique solutions to the equations
 \begin{align*}
(-\Delta)^s u + \sum_{|\alpha|\leq m} a_\alpha D^\alpha u & = 0 \quad \mbox{in} \;\;\Omega,\quad  u - f \in \widetilde H^s(\Omega)  
\end{align*}
and 
\begin{align*}
(-\Delta)^s u^* + \sum_{|\alpha|\leq m} (-1)^{|\alpha|} D^\alpha (a_\alpha  u^*) & = 0 \quad \mbox{in} \;\;\Omega, \quad u^* - f  \in \widetilde H^s(\Omega)
\end{align*}
with $f,g\in H^s(\mathbb R^n)$.
\end{definition}

The next lemma proves that the exterior DN maps $\Lambda_P$ and $\Lambda_P^*$ are well-defined and have some expected properties:

\begin{lemma}[Exterior DN maps]\label{lemma-DN-maps}
The exterior DN maps $\Lambda_P$ and $\Lambda_P^*$ are well-defined, linear and continuous. Moreover, the identity $\langle \Lambda_P[f],[g] \rangle = \langle [f], \Lambda^*_P[g] \rangle$ holds.
\end{lemma}

\begin{proof}
We show well-definedness and continuity only for $\Lambda_P$, the proof being similar for $\Lambda_P^*$. We note that the required unique solutions exist by lemma \ref{well-posedness}.

If $\phi\in \widetilde H^s(\Omega)$, then $u_f|_{\Omega_e} = f = u_{f+\phi}|_{\Omega_e}$, and also $u_f$, $u_{f+\phi}$ both solve $(-\Delta)^su+Pu=0$ in $\Omega$. By unicity of solutions, we must then have that $u_f$ and $u_{f+\phi}$ coincide. On the other hand, if $\psi \in \widetilde H^s(\Omega)$, then $\psi|_{\Omega_e}=0$. These two facts imply the well-definedness of $\Lambda_P$, since $$ B_P(u_{f+\phi},g+\psi) = B_P(u_f,g) + B_P(u_f,\psi) = B_P(u_f,g).$$

The continuity of $\Lambda_P$ is an easy consequence of lemma \ref{boundedness-bilinear} and the estimate in lemma \ref{well-posedness}. If $f,g \in H^s(\mathbb R^n)$ and $\phi,\psi \in \widetilde H^s(\Omega)$, then 
\begin{align*} 
|\langle\Lambda_P[f],[g]\rangle| & = |B_P(u_{f-\phi},g-\psi)| \leq C \|u_{f-\phi}\|_{H^s} \|g-\psi\|_{H^s} \leq C \|f-\phi\|_{H^s} \|g-\psi\|_{H^s}.
\end{align*}
By taking the infimum on both sides with respect to $\phi$ and $\psi$, we end up with 
$$ |\langle\Lambda_P[f],[g]\rangle| \leq C\inf_{\phi\in\widetilde H^s(\Omega)}\|f-\phi\|_{H^s}\inf_{\psi\in\widetilde H^s(\Omega)}\|g-\psi\|_{H^s} = C \|[f]\|_X \|[g]\|_X.$$ 

The well-posedness result proved above implies that for all $f,g \in H^s(\mathbb R^n)$ we have $\langle\Lambda_P[f],[g]\rangle = B_P(u_f, e_g)$, where $e_g$ is a generic extension of $g|_{\Omega_e}$ from $\Omega_e$ to $\mathbb R^n$. In particular, $\langle\Lambda_P[f],[g]\rangle = B_P(u_f, u_g^*)$. By lemma \ref{boundedness-bilinear} this leads to 
$$\langle \Lambda_P[f],[g] \rangle = B_P(u_f,u_g^*) = B_P^*(u_g^*,u_f) = \langle \Lambda^*_P[g],[f] \rangle,$$
which conlcudes the proof.
\end{proof}

\begin{remark}
We should observe at this point that a priori $\Lambda_P^*$ has no reason to be the adjoint of $\Lambda_P$, as the symbols would suggest. However, the identity we proved in lemma \ref{lemma-DN-maps} shows that this is in fact true, and thus there is no abuse of notation.
\end{remark}

\subsection{Proof of injectivity}\label{subsec-injectivity-singular}

The proof of injectivity is based on an Alessandrini identity and the Runge approximation property for our operator, following the scheme developed in \cite{GSU20}. 

\begin{lemma}[Alessandrini identity]\label{alex}
Let $\Omega \subset \mathbb R^n$ be a bounded open set. Let $s \in \mathbb R^+ \setminus \mathbb Z$ and $m\in \mathbb N$ such that $2s > m$. For $j=1,2$, let $a_{j,\alpha} \in M_0(H^{s-|\alpha|}\rightarrow H^{-s})$. For any $f_1, f_2 \in H^s(\mathbb R^n)$, let $u_1,u_2^* \in H^s(\mathbb R^n)$ respectively solve
 \begin{align*}
(-\Delta)^s u_1 + \sum_{|\alpha|\leq m} a_{1,\alpha} D^\alpha u_1 & = 0 \quad \mbox{in} \;\;\Omega,\quad  u_1 - f_1 \in \widetilde H^s(\Omega)  
\end{align*}
and 
\begin{align*}
(-\Delta)^s u_2^* + \sum_{|\alpha|\leq m} (-1)^{|\alpha|} D^\alpha (a_{2,\alpha}  u_2^*) & = 0 \quad \mbox{in} \;\;\Omega, \quad u_2^* - f_2  \in \widetilde H^s(\Omega) .
\end{align*}
Then we have the integral identity
$$ \langle (\Lambda_{P_1} - \Lambda_{P_2})[f_1],[f_2] \rangle = \sum_{|\alpha|\leq m} \langle(a_{1,\alpha}-a_{2,\alpha}), (D^\alpha u_1) u_2^*\rangle. $$
\end{lemma}

\begin{proof}
The proof is a simple computation following from lemma \ref{lemma-DN-maps}:
\begin{align*}
\langle (\Lambda_{P_1} - \Lambda_{P_2})[f_1],[f_2] \rangle & = \langle \Lambda_{P_1} [f_1],[f_2] \rangle - \langle \Lambda_{P_2}[f_1],[f_2] \rangle = \langle \Lambda_{P_1} [f_1],[f_2] \rangle - \langle [f_1], \Lambda_{P_2}^*[f_2] \rangle \\ & = B_{P_1}(u_1,u_2^*) - B_{P_2}^*(u_2^*,u_1) = \sum_{|\alpha|\leq m} \langle(a_{1,\alpha}-a_{2,\alpha}), (D^\alpha u_1) u_2^*\rangle.\qedhere
\end{align*}
\end{proof}

\begin{lemma}[Runge approximation property]\label{runge}
Let $\Omega, W \subset \mathbb R^n$ respectively be a bounded open set and a non-empty open set such that $\overline W \cap \overline \Omega = \emptyset$. Let $s \in \mathbb R^+ \setminus \mathbb Z$ and $m\in \mathbb N$ be such that $2s > m$, and let $a_{\alpha} \in M_0(H^{s-|\alpha|}\rightarrow H^{-s})$. Moreover, let $\mathcal{R} := \{\, u_f - f: f\in C^\infty_c(W) \,\} \subset \widetilde H^s(\Omega)$ where $u_f$ solves
\begin{align*}
(-\Delta)^s u_f + \sum_{|\alpha|\leq m} a_\alpha D^\alpha u_f & = 0 \quad \mbox{in} \;\;\Omega,\quad  u_f - f \in \widetilde H^s(\Omega)  
\end{align*}
and $\mathcal{R}^* := \{\, u^*_f - f: f\in C^\infty_c(W) \,\} \subset \widetilde H^s(\Omega)$ where $u^*_f$ solves
\begin{align*}
(-\Delta)^s u_f^* + \sum_{|\alpha|\leq m} (-1)^{|\alpha|} D^\alpha (a_\alpha  u_f^*) & = 0 \quad \mbox{in} \;\;\Omega, \quad u_f^* - f  \in \widetilde H^s(\Omega) .
\end{align*}
Then $\mathcal R$ and $\mathcal R^*$ are dense in $\widetilde H^s(\Omega)$.
\end{lemma}

\begin{proof}
The proofs of the two statements are similar, so we show only the density of $\mathcal R$ in $\widetilde H^s(\Omega)$. By the Hahn-Banach theorem, it is enough to prove that any functional $F$ acting on $\widetilde H^s(\Omega)$ that vanishes on $\mathcal R$ must be identically $0$. Thus, let $F \in (\widetilde H^s(\Omega))^*$ and assume $F(u_f-f)=0$ for all $f\in C^\infty_c(W)$. Let $\phi$ be the unique solution of
\begin{equation}
\label{eq:rungesolution}
(-\Delta)^s \phi + \sum_{|\alpha|\leq m} (-1)^{|\alpha|} D^\alpha (a_\alpha  \phi)  = -F \quad \mbox{in} \;\;\Omega, \quad \phi \in \widetilde H^s(\Omega).
\end{equation}
In other words, $\phi$ is the unique function in $\widetilde H^s(\Omega)$ such that $B_P^*(\phi,w)=-F(w)$ for all $w\in \widetilde H^s(\Omega)$. Then we can compute
\begin{align}\label{runge-1}
0 & = F(u_f-f) = -B_P^*(\phi, u_f-f) = B_P^*(\phi,f) \\ & = \langle (-\Delta)^{s/2} f, (-\Delta)^{s/2} \phi \rangle + \sum_{|\alpha|\leq m} \langle a_\alpha, D^\alpha f \phi\rangle \\ & = \langle f, (-\Delta)^{s} \phi \rangle.
\end{align}
On the first line of \eqref{runge-1} we used that $\phi \in \widetilde H^s(\Omega)$ and $u_f$ solves the equation in $\Omega$, and on the last line we used the support condition for $f$. By the arbitrariety of $f\in C^\infty_c(W)$ we have obtained that $(-\Delta)^s\phi=0$ in $W$, and on the same set we also have $\phi=0$. Using the unique continuation result for the higher order fractional Laplacian given in lemma \ref{lemma:ucpoffractionallaplacian} we deduce $\phi\equiv 0$ on all of $\mathbb R^n$. The vanishing of the functional $F$ now follows easily from the definition of
~$\phi$.
\end{proof}

\begin{remark} We remark that using the same proof one can show that $r_\Omega\mathcal{R}\subset L^2(\Omega)$ and $r_\Omega\mathcal{R}^*\subset L^2(\Omega)$ are dense in $L^2(\Omega)$, where $r_\Omega$ is the restriction to $\Omega$. If $F\in L^2(\Omega)$, then $F$ induces an element in $(\widetilde{H}^s(\Omega))^*$ via the integral $F(w):=\ip{F}{r_\Omega w}_{L^2(\Omega)}$, where $w\in\widetilde{H}
^s(\Omega)$. Hence one can choose the solution $\phi$ in equation \eqref{eq:rungesolution} with $F$ as a source term and complete the proof as in equation~\eqref{runge-1} showing that $(r_\Omega\mathcal{R})^\perp=\{0\}$ in $L^2(\Omega)$ (similarly $(r_\Omega\mathcal{R}^*)^\perp=\{0\}$). 
\end{remark}

We are ready to prove the main result of the paper.

\begin{proof}[Proof of theorem \ref{main-theorem-singular}]
\textbf{Step 1.} Since one can always shrink the sets $W_1$ and $W_2$ if necessary, we can assume without loss of generality that $\overline{W_1}\cap\overline{W_2}=\emptyset$. Let $v_1, v_2 \in C^\infty_c(\Omega)$. By the Runge approximation property proved in lemma \ref{runge} we can find two sequences of functions $\{f_{j,k}\}_{k\in\mathbb N} \subset C^\infty_c(W_j)$, $j=1,2$, such that 
\begin{align*}
 u_{1,k} = f_{1,k}+v_1+r_{1,k},\quad u^*_{2,k} = f_{2,k}+v_2+r_{2,k}
\end{align*}
where $u_{1,k},\, u_{2,k}^* \in \widetilde H^s(\Omega)$ respectively solve
 \begin{align*}
(-\Delta)^s u_{1,k} + \sum_{|\alpha|\leq m} a_{1,\alpha} D^\alpha u_{1,k} & = 0 \quad \mbox{in} \;\;\Omega,\quad  u_{1,k}-f_{1,k} \in \widetilde H^s(\Omega)  
\end{align*}
and 
\begin{align*}
(-\Delta)^s u_{2,k}^* + \sum_{|\alpha|\leq m} (-1)^{|\alpha|} D^\alpha (a_{2,\alpha}u_{2,k}^*) & = 0 \quad \mbox{in} \;\;\Omega, \quad u_{2,k}^* -f_{2,k} \in \widetilde H^s(\Omega)
\end{align*}
and $r_{1,k},\, r_{2,k}\rightarrow 0$ in $\widetilde H^s(\Omega)$ as $k\rightarrow \infty$. By the assumption on the exterior DN maps and the Alessandrini identity from lemma \ref{alex} we have
\begin{align}\label{main-1} 
0 & = \langle (\Lambda_{P_1} - \Lambda_{P_2})[f_{1,k}],[f_{2,k}] \rangle = \sum_{|\alpha|\leq m} \langle(a_{1,\alpha}-a_{2,\alpha}), (D^\alpha u_{1,k}) u_{2,k}^*\rangle. 
\end{align}

On the other hand, the support conditions imply that
\begin{align}
    \sum_{|\alpha|\leq m} \langle(a_{1,\alpha}-a_{2,\alpha}), (D^\alpha u_{1,k}) u_{2,k}^*\rangle &= \sum_{|\alpha|\leq m} \langle(a_{1,\alpha}-a_{2,\alpha}), (D^\alpha (u_{1,k}-f_{1,k})) (u_{2,k}^*-f_{2,k})\rangle \\
    &= \sum_{|\alpha|\leq m} \langle(a_{1,\alpha}-a_{2,\alpha}), (D^\alpha (v_1+r_{1,k})) (v_2+r_{2,k})\rangle.
\end{align}
Thus by taking the limit $k \to \infty$ and using lemma \ref{boundedness-bilinear}, we obtain
\begin{equation}\label{main-2}
 \sum_{|\alpha|\leq m} \langle(a_{1,\alpha}-a_{2,\alpha}), (D^\alpha v_1) v_2\rangle = 0 \quad \mbox{for all} \quad v_1,v_2 \in C^\infty_c(\Omega)
\end{equation}
by formula \eqref{main-1}.

\textbf{Step 2.} Assume that we have $a_{1,\alpha}|_{\Omega} = a_{2,\alpha}|_{\Omega}$ for all $\alpha$ such that $|\alpha| < N$ for some $N \in \mathbb N$. We show that the equality of the coefficients also holds for $\alpha$ for which $|\alpha| = N$, and this will prove the theorem by the principle of complete induction. 

To this end, consider $v_2 \in C^\infty_c(\Omega)$, and then take $v_1\in C^\infty_c(\Omega)$ such that $v_1(x) = x^\alpha$ on supp$(v_2) \Subset \Omega$. Recall that since $\alpha = (\alpha_1, \alpha_2, ..., \alpha_n) \in \mathbb N^n$ is a multi-index and $x=(x_1, x_2, ..., x_n)\in \mathbb R^n$, the symbol $x^\alpha$ is intended to mean $x_1^{\alpha_1}x_2^{\alpha_2}...\,x_n^{\alpha_n}$. With this choice of $v_1, v_2$, equation \eqref{main-2} becomes

\begin{align}\label{main-3}
0 & = \sum_{|\beta|\leq m} \langle(a_{1,\beta}-a_{2,\beta}), (D^\beta v_1) v_2\rangle = \sum_{N \leq |\beta|\leq m} \langle(a_{1,\beta}-a_{2,\beta}), (D^\beta x^\alpha)v_2\rangle \\ & = \sum_{N < |\beta|\leq m} \langle(a_{1,\beta}-a_{2,\beta}),(D^\beta x^\alpha) v_2\rangle  + \sum_{|\beta|=N, \,\beta \neq \alpha} \langle(a_{1,\beta}-a_{2,\beta}),(D^\beta x^\alpha) v_2\rangle \\ & \quad + \langle(a_{1,\alpha}-a_{2,\alpha}),(D^\alpha x^\alpha) v_2\rangle.
\end{align}

\noindent If $|\beta| > N = |\alpha|$, then there must exist $k \in \{1,2,...,n\}$ such that $\beta_k > \alpha_k$. This is true also if $|\beta| = N$ with $\beta\neq\alpha$. In both cases we can compute
$$ D^\beta(x^\alpha) = (\partial_{x_1}^{\beta_1} x_1^{\alpha_1})\;(\partial_{x_2}^{\beta_2} x_2^{\alpha_2})\; ... \;( \partial_{x_n}^{\beta_n} x_n^{\alpha_n}) = 0 $$
because $\partial_{x_k}^{\beta_k} x_k^{\alpha_k} =0$. Therefore formula \eqref{main-3} becomes
$$ 0=\langle(a_{1,\alpha}-a_{2,\alpha}),(D^\alpha x^\alpha) v_2\rangle = \alpha! \langle a_{1,\alpha}-a_{2,\alpha}, v_2 \rangle $$

\noindent which by the arbitrariety of $v_2\in C^\infty_c(\Omega)$ implies $a_{1,\alpha}|_{\Omega} = a_{2,\alpha}|_{\Omega}$ also for $\alpha$ for which $|\alpha| = N$. 

\textbf{Step 3.} We have proved that $a_{1,\alpha}|_{\Omega} = a_{2,\alpha}|_{\Omega}$ for all $\alpha$ of order $|\alpha| \leq m$. Since this entails $P_1|_{\Omega} = P_2|_{\Omega}$, the proof is complete.
\end{proof}

\section{Main theorem for bounded coefficients}\label{section-main-bounded}

We shall now study the case when the coefficients of PDOs are from the bounded spaces $H^{r_\alpha,\infty}(\Omega)$. It should be noted, however, that most of the considerations of the previous section still apply identically. 

\subsection{Well-posedness of the forward problem}\label{subsec-WP-bounded}

We shall define the bilinear forms for the problems \eqref{problem} and \eqref{adjoint-problem} respectively by \eqref{bilinear} and \eqref{adjoint-bilinear}, just as in the case of singular coefficients. These will turn out to be bounded in $H^s(\R^n)\times H^s(\R^n)$ as well, but the proof we give of this fact is \emph{a fortiori} different. Since now we assume that $a_\alpha\in H^{r_\alpha, \infty}(\Omega)\subset L^\infty(\Omega)$ for $r_\alpha\geq 0$, the duality pairing $\ip{a_\alpha}{(D^\alpha v) w}$ becomes an integral over $\Omega$.

\begin{lemma}[Boundedness of the bilinear forms]\label{boundedness-bilinear-bounded}
Let $\Omega \subset \mathbb R^n$ be a bounded Lipschitz domain and $s \in \mathbb R^+ \setminus \mathbb Z$, $m\in \mathbb N$ such that $2s > m$. Let $a_{\alpha} \in H^{r_\alpha,\infty}(\Omega)$, with $r_\alpha$ defined as in \eqref{r-alpha-conditions}. Then $B_P$ and $B_P^*$ extend as bounded bilinear forms on $H^s(\mathbb R^n)\times H^s(\mathbb R^n)$.
\end{lemma}

\begin{remark}\label{r-alpha-remark}
\noindent Since $s \in  \mathbb R^+ \setminus \mathbb Z$ and $|\alpha| \leq m < 2s$, we also have that $\max(0, |\alpha|-s) \leq r_\alpha < s$ for $\delta > 0$ small (see formula \eqref{r-alpha-conditions}).
\end{remark}

\begin{proof}[Proof of lemma \ref{boundedness-bilinear-bounded}]
We only prove the boundedness of $B_P$, as for $B_P^*$ one can proceed in the same way. If $v,w \in C^\infty_c(\mathbb R^n)$, then
\begin{align}\label{estimate-PDO-1-bounded}
|\langle a_\alpha(x) D^\alpha v, w \rangle| & = \left|\int_\Omega a_\alpha w (D^\alpha v) \, dx\right| \leq \|a_\alpha w\|_{(H^{-r_\alpha}(\Omega))^*}\|D^\alpha v\|_{H^{-r_\alpha}(\Omega)}.
\end{align}
Since $\Omega$ is a Lipschitz domain and $r_\alpha \geq 0$, $r_\alpha \not\in \left\{\frac{1}{2}, \frac{3}{2}, \frac{5}{2} ...\right\}$, we have $(H^{-r_\alpha}(\Omega))^*= H^{r_\alpha}_0(\Omega)\subset H^{r_\alpha}(\Omega)$. Therefore
\begin{align}\label{estimate-PDO-1.5-bounded}
|\langle a_\alpha(x) D^\alpha v, w \rangle| & \leq C \|a_\alpha w\|_{ H^{r_\alpha}(\Omega)}\|D^\alpha v\|_{H^{-r_\alpha}(\Omega)} \leq C \|A_\alpha w\|_{ H^{r_\alpha}(\mathbb R^n)}\|D^\alpha v\|_{H^{-r_\alpha}(\Omega)} \\ 
& \leq C \|J^{r_\alpha}(A_\alpha w)\|_{L^2(\mathbb R^n)} \|v\|_{H^{|\alpha|-r_\alpha}(\Omega)}
\end{align}

\noindent where $J= (\id-\Delta)^{1/2}$ is the Bessel potential and $A_\alpha$ is an extension of $a_\alpha$ from $\Omega$ to $\mathbb R^n$ such that $A_\alpha|_\Omega = a_\alpha$ and $\|A_\alpha\|_{H^{r_\alpha,\infty}(\mathbb R^n)} \leq 2 \|a_\alpha\|_{H^{r_\alpha,\infty}(\Omega)}$. Since $r_\alpha \geq 0$, we may estimate the last term of \eqref{estimate-PDO-1.5-bounded} by the Kato-Ponce inequality given in lemma \ref{lemma:katoponce}
\begin{align}
\|J^{r_\alpha}(A_\alpha w)\|_{L^2(\mathbb R^n)} & \leq C \left( \|A_\alpha\|_{L^\infty(\mathbb R^n)} \|J^{r_\alpha} w\|_{L^2(\mathbb R^n)} + \|J^{r_\alpha} A_\alpha\|_{L^\infty(\mathbb R^n)} \|w\|_{L^2(\mathbb R^n)} \right) \\
& \leq C \|A_\alpha\|_{H^{r_\alpha,\infty}(\mathbb R^n)} \|w\|_{H^{r_\alpha}(\mathbb R^n)} \leq C \|a_\alpha\|_{H^{r_\alpha,\infty}(\Omega)} \|w\|_{H^{r_\alpha}(\mathbb R^n)}.
\end{align}

\noindent Substituting this into \eqref{estimate-PDO-1.5-bounded} gives
\begin{align}\label{estimate-PDO-2}
|\langle a_\alpha(x) D^\alpha v, w \rangle| &  \leq C  \|a_\alpha\|_{H^{r_\alpha,\infty}(\Omega)} \|w\|_{H^{r_\alpha}(\mathbb R^n)} \|v\|_{H^{|\alpha|-r_\alpha}(\Omega)} \\ 
& \leq C  \|a_\alpha\|_{H^{r_\alpha,\infty}(\Omega)} \|w\|_{H^s(\mathbb R^n)} \|v\|_{H^s(\mathbb R^n)}
\end{align}
\noindent given that both $r_\alpha < s$ and $|\alpha| - r_\alpha \leq s$ hold by remark \ref{r-alpha-remark}. Eventually we obtain

\begin{align}
|B_P(v,w)| & \leq |\langle (-\Delta)^{s/2}v, (-\Delta)^{s/2}w \rangle | + \sum_{|\alpha|\leq m} |\langle a_\alpha D^\alpha v, w \rangle| \\
& \leq \|w\|_{H^s(\mathbb R^n)} \|v\|_{H^s(\mathbb R^n)} + \sum_{|\alpha|\leq m} C  \|a_\alpha\|_{H^{r_\alpha,\infty}(\Omega)} \|w\|_{H^s(\mathbb R^n)} \|v\|_{H^s(\mathbb R^n)} \\ & 
\leq C \|w\|_{H^s(\mathbb R^n)} \|v\|_{H^s(\mathbb R^n)}.\qedhere
\end{align}
\end{proof}

Next we shall prove existence and uniqueness of solutions for the problems \eqref{problem} and \eqref{adjoint-problem}. The reasoning is similar to the one for the proof of lemma \ref{well-posedness}, but the details of the computations are quite different.  

\begin{lemma}[Well-posedness]\label{well-posedness-bounded}
Let $\Omega \subset \mathbb R^n$ be a bounded Lipschitz domain and $s \in \mathbb R^+ \setminus \mathbb Z$, $m\in \mathbb N$ such that $2s > m$. Let $a_{\alpha} \in H^{r_\alpha,\infty}(\Omega)$, with $r_\alpha$ defined as in \eqref{r-alpha-conditions}. There exist a real number $\mu > 0$ and a countable set $\Sigma \subset (-\mu, \infty)$ of eigenvalues $\lambda_1 \leq \lambda_2 \leq ... \rightarrow \infty$ such that if $\lambda \in \mathbb R \setminus \Sigma$, for any $f\in H^s(\mathbb R^n)$ and $F\in (\widetilde{H}
^s(\Omega))^*$ there exists a unique $u\in H^s(\mathbb R^n)$ such that $u-f \in \widetilde H^s(\Omega)$ and
$$ B_P(u,v)-\lambda \langle u,v \rangle = F(v)\quad \mbox{for all} \quad v \in \widetilde H^s(\Omega).$$
One has the estimate 
$$ \|u\|_{H^s(\mathbb R^n)} \leq C\left( \|f\|_{H^s(\mathbb R^n)} + \|F\|_{(\widetilde{H}^s(\Omega))^*} \right). $$
The function $u$ is also the unique $u\in H^s(\mathbb R^n)$ satisfying $$r_\Omega \left( (-\Delta)^s + \sum_{|\alpha|\leq m} a_\alpha(x) D^\alpha  -\lambda \right)u=F$$ in the sense of distributions in $\Omega$ and $u-f \in \widetilde H^s(\Omega)$. Moreover, if \eqref{eigenvalue-condition-1} holds then $0\notin \Sigma$.
\end{lemma}

\begin{proof}
Again it is enough to find unique $\tilde u \in \widetilde H^s(\Omega)$ such that $B_P(\tilde u, v) - \lambda \langle \tilde u,v \rangle = \tilde F(v)$, where $\tilde F := F - B_P(f,\cdot) + \lambda\langle f,\cdot\rangle$. Consider $v,w \in C^\infty_c(\Omega)$ and $r_\alpha \neq 0$. Since $0 < r_\alpha< s$, the interpolation inequality $$ \|w\|_{H^{r_\alpha}(\mathbb R^n)} \leq C \|w\|_{L^2(\mathbb R^n)}^{1-r_{\alpha}/s} \|w\|_{H^s(\mathbb R^n)}^{r_\alpha/s} $$ holds. Using this and formula \eqref{estimate-PDO-2} we get, for a constant $C=C(\Omega,n,s,r_\alpha)$ which may change from line to line,
\begin{align}\label{WP-1}
|\langle a_\alpha(x) D^\alpha v, w \rangle| &  \leq C  \|a_\alpha\|_{H^{r_\alpha,\infty}(\Omega)}\|v\|_{H^s(\mathbb R^n)}  \|w\|_{H^{r_\alpha}(\mathbb R^n)}  \\ 
& \leq C  \|a_\alpha\|_{H^{r_\alpha,\infty}(\Omega)} \|v\|_{H^s(\mathbb R^n)} \|w\|_{L^2(\mathbb R^n)}^{1-r_{\alpha}/s} \|w\|_{H^s(\mathbb R^n)}^{r_\alpha/s}  \\
& \leq \|a_\alpha\|_{H^{r_\alpha,\infty}(\Omega)} \|v\|_{H^s(\mathbb R^n)} \left(  C \epsilon^{r_\alpha/(r_\alpha-s)} \|w\|_{L^2(\mathbb R^n)} + \epsilon \|w\|_{H^s(\mathbb R^n)} \right).
\end{align}

\noindent In the last step of \eqref{WP-1} we used formula \eqref{young-general} with 
$$q=\frac{s}{r_\alpha}, \quad p=\frac{s}{s-r_\alpha}, \quad b= \|w\|_{H^s(\mathbb R^n)}^{r_\alpha/s}, \quad a= C\|w\|_{L^2(\mathbb R^n)}^{1-r_\alpha/s}, \quad \eta = \epsilon.$$ 

If instead $r_\alpha =0$, just by formula  \eqref{estimate-PDO-2} we already have $$ |\langle a_\alpha(x) D^\alpha v, w \rangle| \leq C  \|a_\alpha\|_{L^{\infty}(\Omega)}\|v\|_{H^s(\mathbb R^n)}  \|w\|_{L^2(\mathbb R^n)} .$$

Moreover, the two estimates above also hold for $v,w\in \widetilde H^s(\Omega)$ by the density of $C^\infty_c(\Omega)$ in $\widetilde H^s(\Omega)$. Now we use formula \eqref{young-general} again, but this time we choose
$$q=p=2, \quad b= \|v\|_{H^s(\mathbb R^n)}, \quad a= \|v\|_{L^2(\mathbb R^n)}, \quad \eta = \epsilon^{s/(s-r_\alpha)}.$$ 
This leads to
\begin{align}\label{WP-2}
    | \langle a_\alpha(x) D^\alpha v, v \rangle | & \leq \|a_\alpha\|_{H^{r_\alpha,\infty}(\Omega)} \|v\|_{H^s(\mathbb R^n)} \left(  C \epsilon^{r_\alpha/(r_\alpha-s)} \|v\|_{L^2(\mathbb R^n)} + \epsilon \|v\|_{H^s(\mathbb R^n)} \right) \\
    & = \|a_\alpha\|_{H^{r_\alpha,\infty}(\Omega)} \left( C \epsilon^{r_\alpha/(r_\alpha-s)} \|v\|_{L^2(\mathbb R^n)}\|v\|_{H^s(\mathbb R^n)} + \epsilon \|v\|_{H^s(\mathbb R^n)}^2 \right) \\ 
    & \leq \|a_\alpha\|_{H^{r_\alpha,\infty}(\Omega)} \left( C \epsilon^{\frac{r_\alpha+s}{r_\alpha-s}} \|v\|_{L^2(\mathbb R^n)}^2 + \epsilon (C+1) \|v\|_{H^s(\mathbb R^n)}^2 \right) \\
    & \leq C \|a_\alpha\|_{H^{r_\alpha,\infty}(\Omega)} \left( \epsilon^{\frac{r_\alpha+s}{r_\alpha-s}} \|v\|_{L^2(\mathbb R^n)}^2 + \epsilon \|v\|_{H^s(\mathbb R^n)}^2 \right) \\
    & \leq C' \|a_\alpha\|_{H^{r_\alpha,\infty}(\Omega)} \left( \epsilon^{\frac{M+s}{M-s}} \|v\|_{L^2(\mathbb R^n)}^2 + \epsilon \|v\|_{H^s(\mathbb R^n)}^2 \right) 
\end{align}
\noindent where $C=C(\Omega,n,s,r_\alpha)$ and $C'=C'(\Omega,n,s)$ are constants changing from line to line and $M\in[0,s)$ is defined by $M:= \max_{|\alpha|\leq m} r_\alpha$. Eventually
\begin{align}\label{coercivity-1}
B_P(v,v) & \geq \|(-\Delta)^{s/2}v\|^2_{L^2(\mathbb R^n)} - \sum_{|\alpha|\leq m} | \langle a_\alpha(x) D^\alpha v, v \rangle | \\ 
& \geq  \|(-\Delta)^{s/2}v\|^2_{L^2(\mathbb R^n)} -  C' \left( \epsilon^{\frac{M+s}{M-s}} \|v\|_{L^2(\mathbb R^n)}^2 + \epsilon \|v\|_{H^s(\mathbb R^n)}^2 \right)  \sum_{|\alpha|\leq m}  \|a_\alpha\|_{H^{r_\alpha,\infty}(\Omega)} \\
& = \|(-\Delta)^{s/2}v\|^2_{L^2(\mathbb R^n)} - C' C'' \left( \epsilon^{\frac{M+s}{M-s}} \|v\|_{L^2(\mathbb R^n)}^2 + \epsilon \|v\|_{H^s(\mathbb R^n)}^2 \right)
\end{align}
\noindent where $C'' := \sum_{|\alpha|\leq m} \|a_\alpha\|_{H^{r_\alpha,\infty}(\Omega)}$ is a constant independent of $\epsilon$ and $v$. By the higher order Poincar\'e inequality (lemma \ref{lemma:fractionalpoincareinequliaty}) \eqref{coercivity-1} turns into
\begin{align}
B_P(v,v) & \geq c \left(\|(-\Delta)^{s/2}v\|^2_{L^2(\mathbb R^n)}+ \|v\|_{L^2(\mathbb R^n)}^2\right)  - C' C'' \left( \epsilon^{\frac{M+s}{M-s}} \|v\|_{L^2(\mathbb R^n)}^2 + \epsilon \|v\|_{H^s(\mathbb R^n)}^2 \right) \\ 
& \geq c \|v\|_{H^s(\mathbb R^n)}^2 - C' C'' \left( \epsilon^{\frac{M+s}{M-s}} \|v\|_{L^2(\mathbb R^n)}^2 + \epsilon \|v\|_{H^s(\mathbb R^n)}^2 \right) 
\end{align}
\noindent for some constant $c=c(\Omega,n,s)$ changing from line to line. For $\epsilon$ small enough (notice that $M-s < 0$), this eventually gives the coercivity estimate
\begin{equation}\label{coercivity-3}
B_P(v,v) \geq c_0 \|v\|_{H^s(\mathbb R^n)}^2 - \mu\|v\|_{L^2(\mathbb R^n)}^2
\end{equation}
\noindent for some constants $c_0, \mu > 0$ independent of $v$. The proof is now concluded as in lemma \ref{well-posedness}.
\end{proof}

Assuming as in Section \ref{section-main-sigular} that both \eqref{eigenvalue-condition-1} and \eqref{eigenvalue-condition-2} hold, by means of the above lemma \ref{well-posedness-bounded} we can define the DN-maps $\Lambda_P, \Lambda_P^*$ just as in lemma \ref{lemma-DN-maps}.

\begin{definition}
Let $\Omega \subset \mathbb R^n$ be a bounded open set. Let $s \in \mathbb R^+ \setminus \mathbb Z$ and $m\in \mathbb N$ such that $2s> m$, and let $a_{\alpha} \in H^{r_\alpha,\infty}(\Omega)$, with $r_\alpha$ defined as in \eqref{r-alpha-conditions}. The exterior DN maps $\Lambda_P$ and $\Lambda_P^*$ are  $$ \Lambda_P : X\rightarrow X^* \quad \mbox{defined by}\quad \langle \Lambda_P[f],[g] \rangle := B_P(u_f, g)$$
and 
$$ \Lambda^*_P : X\rightarrow X^* \quad \mbox{defined by}\quad \langle \Lambda^*_P[f],[g] \rangle := B^*_P(u^*_f, g)$$
where $u_f, u^*_f$ are the unique solutions to the equations
 \begin{align*}
(-\Delta)^s u + \sum_{|\alpha|\leq m} a_\alpha D^\alpha u & = 0 \quad \mbox{in} \;\;\Omega,\quad  u - f \in \widetilde H^s(\Omega)  
\end{align*}
and 
\begin{align*}
(-\Delta)^s u^* + \sum_{|\alpha|\leq m} (-1)^{|\alpha|} D^\alpha (a_\alpha  u^*) & = 0 \quad \mbox{in} \;\;\Omega, \quad u^* - f  \in \widetilde H^s(\Omega)
\end{align*}
with $f,g\in H^s(\mathbb R^n)$.
\end{definition}

\subsection{Proof of injectivity}\label{subsec-injectivity-bounded}

 We also arrive at the same Alessandrini identity and Runge approximation property which we get in lemmas \ref{alex} and \ref{runge}.

\begin{lemma}[Alessandrini identity]\label{alex-bounded}
Let $\Omega \subset \mathbb R^n$ be a bounded Lipschitz domain and $s \in \mathbb R^+ \setminus \mathbb Z$, $m\in \mathbb N$ such that $2s > m$. Let $a_{\alpha} \in H^{r_\alpha,\infty}(\Omega)$, with $r_\alpha$ defined as in \eqref{r-alpha-conditions}. For any $f_1, f_2 \in H^s(\mathbb R^n)$, let $u_1,u_2^* \in H^s(\mathbb R^n)$ respectively solve
 \begin{align*}
(-\Delta)^s u_1 + \sum_{|\alpha|\leq m} a_{1,\alpha}(x) D^\alpha u_1 & = 0 \quad \mbox{in} \;\;\Omega,\quad  u_1 - f_1 \in \widetilde H^s(\Omega)  
\end{align*}
and 
\begin{align*}
(-\Delta)^s u_2^* + \sum_{|\alpha|\leq m} (-1)^{|\alpha|} D^\alpha (a_{2,\alpha}(x)  u_2^*) & = 0 \quad \mbox{in} \;\;\Omega, \quad u_2^* - f_2  \in \widetilde H^s(\Omega) .
\end{align*}
Then we have the integral identity
$$ \langle (\Lambda_{P_1} - \Lambda_{P_2})[f_1],[f_2] \rangle = \sum_{|\alpha|\leq m} \langle(a_{1,\alpha}-a_{2,\alpha})D^\alpha u_1, u_2^*\rangle. $$
\end{lemma}

\begin{lemma}[Runge approximation property]\label{runge-bounded}
Let $\Omega, W \subset \mathbb R^n$ respectively be a bounded Lipschitz domain and a non-empty open set such that $\overline W \cap \overline \Omega = \emptyset$. Let $s \in \mathbb R^+ \setminus \mathbb Z$, $m\in \mathbb N$ such that $2s > m$. Let $a_{\alpha} \in H^{r_\alpha,\infty}(\Omega)$, with $r_\alpha$ defined as in \eqref{r-alpha-conditions}. Moreover, let $\mathcal{R} := \{\, u_f - f: f\in C^\infty_c(W) \,\} \subset \widetilde H^s(\Omega)$, where $u_f$ solves
\begin{align*}
(-\Delta)^s u_f + \sum_{|\alpha|\leq m} a_\alpha(x) D^\alpha u_f & = 0 \quad \mbox{in} \;\;\Omega,\quad  u_f - f \in \widetilde H^s(\Omega)  
\end{align*}
and $\mathcal{R}^* := \{\, u^*_f - f: f\in C^\infty_c(W) \,\} \subset \widetilde H^s(\Omega)$, where $u^*_f$ solves
\begin{align*}
(-\Delta)^s u_f^* + \sum_{|\alpha|\leq m} (-1)^{|\alpha|} D^\alpha (a_\alpha(x)  u_f^*) & = 0 \quad \mbox{in} \;\;\Omega, \quad u_f^* - f  \in \widetilde H^s(\Omega) .
\end{align*}
Then $\mathcal R$ and $\mathcal R^*$ are dense in $\widetilde H^s(\Omega)$.
\end{lemma}

With this at hand, we can prove the main theorem for bounded coefficients.

\begin{proof}[Proof of theorem \ref{main-theorem-bounded}]
  The proof is virtually identical to the one of theorem \ref{main-theorem-singular}, the unique difference being in the way the error terms of the Runge approximation are estimated. We make use of \eqref{estimate-PDO-2}, which relied on the Kato-Ponce inequality instead of multiplier space estimates. The proof is otherwise completed as the proof of theorem \ref{main-theorem-singular}.
\end{proof}

\bibliography{refs_02} 

\bibliographystyle{abbrv}

\end{document}